\documentclass{article}

\usepackage{geometry}
\usepackage[mathcal]{euscript}
\usepackage{hyperref}
\usepackage{amsmath,amssymb,amsthm}
\usepackage{mathtools}
\DeclareMathAlphabet{\mathpzc}{OT1}{pzc}{m}{it}
\DeclareMathOperator{\dist}{dist}
\usepackage{xcolor}
\usepackage{cases}
\newtheorem{thm}{Theorem}[section]
\newtheorem{lemma}[thm]{Lemma}
\newtheorem{prop}[thm]{Proposition}

\theoremstyle{remark}
\newtheorem{rem}[thm]{Remark}

\theoremstyle{definition}
\newtheorem{defn}[thm]{Definition}
\newtheoremstyle{Claim}{}{}{\itshape}{}{\itshape\bfseries}{:}{ }{#1}
\theoremstyle{Claim}

\newcommand{\Z}{{\mathbb{Z}}}
\newcommand{\T}{{\mathbb{T}}}

\renewcommand{\H}{\mathcal{H}}

\newcommand{\R}{\mathbb{R}}

\newcommand{\N}{\mathbb{N}}

\newcommand{\eps}{\varepsilon}
\newcommand{\norm}[1]{\left\lVert#1\right\rVert}
\DeclareMathOperator{\dive}{div}

\titlepage
\title{On the existence and uniqueness of solutions\\ to time-dependent fractional MFG}
\author{Marco Cirant and Alessandro Goffi}
\date{\today}

\begin{document}

\maketitle

\begin{abstract}
We establish existence and uniqueness of solutions to evolutive fractional Mean Field Game systems with regularizing coupling, for any order of the fractional Laplacian $s\in(0,1)$. The existence is addressed via the vanishing viscosity method. In particular, we prove that in the subcritical regime $s>1/2$ the solution of the system is classical, while if $s\leq 1/2$ we find a distributional energy solution. To this aim, we develop an appropriate functional setting based on parabolic Bessel potential spaces. We show uniqueness of solutions both under monotonicity conditions and for short time horizons.
\end{abstract}

\noindent
{\footnotesize \textbf{AMS-Subject Classification}}. {\footnotesize 35R11, 35K55, 49N70, 35Q84, 91A13.}

\tableofcontents

\section{Introduction}
This paper deals with the following backward-forward coupled system of integro-differential Hamilton-Jacobi-Bellman (HJB) and Fokker-Planck equations
\begin{equation}\label{fmfg}
\begin{cases}
-\partial_tu+(-\Delta)^s u+H(x,Du)=F[m(t)](x) & \text{ in }Q_T\\
\partial_tm+(-\Delta)^s m-\dive(mD_pH(x,Du))=0 & \text{ in }Q_T\\
m(x,0)=m_0(x),\, u(x,T)=u_T(x) & \text{ in }\T^d\ ,
\end{cases}
\end{equation}
where $Q_T:=\T^d\times[0,T]$, $\T^d$ stands for the flat torus $\R^d / \mathbb{Z}^d$, $H=H(x,p)$ is a superlinear Hamiltonian in the second variable, $(-\Delta)^s u$ is the fractional Laplacian of order $s$, $F$ is a regularizing coupling and $m_0,u_T$ are given functions.

Systems of the form \eqref{fmfg} arise in Mean Field Games (briefly MFG) theory, whose goal is to describe the collective behavior of a continuum of rational agents, each of whom seeks to minimize a common criterion. This theory was developed independently by Lasry-Lions \cite{ll} and by Huang et al. \cite{hcm} with the aim of describing Nash equilibria in differential games with infinitely many players. Recently, MFG theory has stimulated an increasing interest due to the wide range of applications in engineering, finance and social sciences among others.

From a PDE viewpoint, the analysis of such models has been carried out either when the dynamics of the average player is driven by standard diffusions (see for example \cite{GomesBook, ll}), possibly degenerate \cite{CGPT}, or first order (deterministic) systems (see e.g. \cite{CG,CPT}). Our purpose is to study an intermediate situation, where the dynamics of agents is perturbed by a $2s$-stable L\'evy process instead of the standard diffusion. L\'evy processes meet a variety of challenging topics ranging from financial modeling (see e.g. the monograph \cite{ContBook}) to physics and biology among others. We refer to \cite{Bertoin,Taqqu} for a comprehensive treatment of stable-like processes, to the monograph \cite{Applebaum} for a more general analysis on jump-type processes and the nice survey \cite{ApplebaumNote}.

The stationary counterpart of \eqref{fmfg}, which heuristically describes an equilibrium state in the long-time regime, has been analyzed very recently by the first author and collaborators \cite{CCDNV}. In particular, in \cite{CCDNV} the investigation is performed for the subcritical order of the fractional Laplacian $s\in(\frac12,1)$, both in the case of local and nonlocal coupling between the equations. There, the well-posedness of the fractional Fokker-Planck equation is based on variational methods, while the study of the fractional HJB equation is established via viscosity solutions' techniques.

Here, we address the existence and uniqueness of solutions to \eqref{fmfg} through the vanishing viscosity method, namely solutions of \eqref{fmfg} are obtained as limits (in some sense to be specified below) of solutions $u_{\sigma}$ of the approximating viscous coupled system of PDEs
\begin{equation}\label{fmfgv}
\begin{cases}
-\partial_tu-\sigma\Delta u+(-\Delta)^s u+H(x,Du)=F[m(t)](x) & \text{ in }Q_T\\
\partial_tm-\sigma\Delta m+(-\Delta)^s m-\dive(mD_pH(x,Du))=0 & \text{ in }Q_T\\
m(x,0)=m_0(x),\,\, u(x,T)=u_T(x) & \text{ in }\T^d\ .
\end{cases}
\end{equation}
Such way to tackle the existence issue for first order systems has been sketched in \cite[Section 4.4]{CardaliaNotes}, and it is a quite natural approach in our setting: \eqref{fmfgv} behaves well in terms of regularity, and is also meaningful from the stochastic viewpoint.

In this paper we provide existence and uniqueness results for any order of the fractional Laplacian $s\in(0,1)$. As it often happens in the PDE literature of MFG, we consider the periodic case, namely all the data are defined on $\T^d$. This is the typical compact setting where one avoids boundary phenomena. While this work was under preparation, we realized that many technical ingredients regarding fractional calculus in the periodic case were not available in the literature, and known at best to few experts. Part of this work is then devoted to provide a self-contained survey on several tools and techniques, ranging from harmonic analysis to interpolation theory, hoping that these may be useful for future research in this area. This material is basically contained in the appendices and at the beginning of Section \ref{sintro}. 

Bessel potential spaces on the torus $H_p^{\mu}(\T^d)$ constitute a natural functional framework for the periodic fractional Laplacian, and can be directly defined through multiple Fourier series. Since we deal with parabolic problems, we also need suitable space-time spaces, on which it is possible to establish (linear) parabolic regularity. Here, one expects space regularity of a solution and of its time derivative to differ by a factor of $2s$. Hence, we systematically treat spaces of the form
\[
\H_p^{\mu}(Q_T) = \H_p^{\mu; s}(\T^d\times (0,T))=\{u\in L^p(0,T;H_p^{\mu}(\T^d))\ ,\partial_t u\in L^{p}(0,T;H_{p}^{\mu-2s}(\T^d))\},
\]
that are clearly reminiscent of classical parabolic Sobolev spaces $W^{2,1}_p$. We prove some fractional parabolic regularity theorems, and chain/product rules that are crucial to work in the nonlinear setting. Then, inspired by some results that appeared in the context of stochastic partial differential equations, we prove an embedding theorem for $\H_p^{\mu}(Q_T)$ that, apart from its own interest, plays a key role in the analysis of \eqref{fmfg}. We refer to \cite{CL} for some discussions on $\H_p^{\mu,s}((0,T) \times \R^d)$, and \cite{KrylovBookSPDE} and references therein for the case $s = 1$.

Let us now enter into a more detailed description of the main results of the paper. First, let us state all the assumptions that will be in force throughout the article. We suppose that $H(x,p)$ is $C^3(\T^d\times\R^d)$, convex in the second variable, $H(x,p)\geq H(x,0)=0$ and there exist constants $\gamma > 1$ and $c_H,C_H,\tilde{C}_H>0$ such that
\begin{align}
\tag{H1}\label{H1} & D_pH(x,p)\cdot p-H(x,p)\geq C_H|p|^{\gamma}-c_H\ , \\
\tag{H2}\label{H2} & H(x,p)-H(x,q)\leq C_H(|p|^{\gamma-1}+|q|^{\gamma-1})|p-q| \, \\
\tag{H3}\label{H3} & |D_{xx}^2H(x,p)|\leq C_H|p|^{\gamma}+\tilde{C}_H \ , \\
\tag{H4}\label{H4} & |D_{px}^2H(x,p)|\leq C_H|p|^{\gamma-1}+\tilde{C}_H \ , \\
\tag{H5}\label{H5} & D_{pp}^2H(x,p)\xi\cdot \xi\geq C_H|p|^{\gamma-2}|\xi|^2-\tilde{C}_H 
\end{align}
for every $x\in \T^d$, $p\in\R^d$ and $\xi\in\R^d$. Denote by $\mathcal{P}(\T^d)$ the set of Borel probability measures on $\T^d$ endowed with the Monge-Kantorovich distance\footnote{$\mathbf{d}_1(\mu,\nu):=\sup_{\varphi}\int_{\T^d}\varphi d(\mu-\nu)$, where the supremum is taken over the 1-Lipschitz maps $\varphi:\T^d\to\R$.} $\mathbf{d}_1$. The following are the standing assumptions on the regularizing coupling $F$: there exists a constant $C_F>0$ such that
\begin{align}
\tag{F1}\label{F1} & \text{$F:\mathcal{P}(\T^d) \rightarrow C^{2+\alpha}(\T^d)$ is continuous,} \\
\tag{F2}\label{F2} & \text{$\|F[m_1]-F[m_2]\|_{C^{2+\alpha}(\T^d)} \leq C_F\mathbf{d}_1(m_1,m_2)$ for all $m_1,m_2\in \mathcal{P}(\T^d)$}, \\
\tag{F3}\label{F3} & \text{$\| F(\cdot,m) \|_{C^{2+\alpha}(\T^d)}\leq C_F$ for every $m\in  \mathcal{P}(\T^d)$.}
\end{align}
Finally, we suppose that
\begin{equation}
\tag{I}\label{I}\text{$u_T\in C^{4+\alpha}(\T^d)$, \ $m_0\in C^{4+\alpha}(\T^d)$ is non-negative and $\int_{\T^d}m_0(x)dx=1$}.
  \end{equation}

 As announced, our first step is to construct solutions of the viscous coupled system \eqref{fmfgv}. More precisely, we have the following
\begin{thm}\label{sigmapositivo}
Let \eqref{I}, \eqref{H1}-\eqref{H5} and \eqref{F1}-\eqref{F3} be in force. Then, for all $\sigma > 0$ and $s\in(0,1)$, there exists a classical solution $(u_{\sigma},m_{\sigma})$ to the fractional MFG system \eqref{fmfgv}.
\end{thm}

The proof of this result is a rather standard application of Schauder's fixed point theorem. For fixed $\sigma > 0$, we  treat $(-\Delta)^s u$, $(-\Delta)^s m$ as perturbation terms in a viscous MFG system. Semiconcavity estimates for the HJB equation with mixed local and nonlocal diffusion term are obtained by means of the adjoint method, that ensure existence of $u$. Note that these estimates are stable as $\sigma \to 0$. This limiting procedure is then described by the next main result:

\begin{thm}\label{vanishingviscosity}
Under the same assumptions of Theorem \ref{sigmapositivo}, let $(u_{\sigma},m_{\sigma})$ be a solution to \eqref{fmfgv}. Then, as $\sigma\rightarrow0$ and up to subsequences, $u_{\sigma}$ converges uniformly to $u$, $D u_\sigma$ converges strongly to $D u$, and $m_{\sigma}$ converges weakly to $m$.
If $s\in(0, 1/2]$, then $(u, m)$ is a weak solution to \eqref{fmfg}, and $(u, m) \in \H_p^{2s}(Q_T) \times \H_p^{2s-1}(Q_T)$ for all $p \in (1, \infty)$.
If $s\in(1/2,1)$, then $\partial_tu,\partial_tm,(-\Delta)^s u,(-\Delta)^s m$ belong to some $\mathcal{C}^{\bar \alpha,\frac {\bar \alpha} {2s}}(Q_T)$, $\bar \alpha \in (0,1)$, and $(u,m )$ is classical solution to \eqref{fmfg} .
\end{thm}

For a more complete statement of convergences of $u_\sigma, D u_\sigma, m_\sigma$, see (i)-(vi) at the beginning of the Proof of Theorem \ref{vanishingviscosity} in Section \ref{SecVV}. Moreover, we refer to the weak notion of solution as the energy one, as detailed in Definitions \ref{defFP} and \ref{weak}. We mention that very little is known about fractional Fokker-Planck equations, so part of Section \ref{SecFP} is devoted to establish some basic facts and properties of solutions. The weak treatment of Fokker-Planck equations with local non-degenerate diffusion (see e.g. \cite{P} and references therein) cannot be directly converted to the nonlocal framework, heuristically because of the gap between the energy terms $(-\Delta)^{s/2}$ and the divergence term. Thus, first order techniques as the ones described in \cite{WeiTian} for the euclidean case are better suited to work in the nonlocal setting.

Regarding uniqueness of solutions, we recall that it is known to hold under two different regimes for MFG driven by local diffusions. The first one requires monotonicity of $F$ and convexity of $H$, and appeared in the seminal papers by Lasry-Lions, while the second one is when the time horizon $T$ is small. The latter was formally presented in the recorded lectures of Lions, and it has been re-analyzed recently in the literature (see \cite{BC,BF,CT}). The monotone case carries over in our fractional framework, as we have enough regularity of $u,m$ and uniqueness for the equations by simple energy arguments. As for the short-time regime, the proofs proposed in \cite{BC,BF} cannot be adapted to our setting, being designed for the Laplacian and established through $L^2$-type estimates. Here, we follow an approach presented in \cite{CGM, CT} to deal with the existence problem in the local case. The idea will be to exploit decay properties of the semigroup associated to the fractional Laplacian in suitable Bessel potentials spaces. These will be strong enough only in the case $s > 1/2$. We stress that here it is crucial to have fractional product (also known as Kato-Ponce inequalities) and chain rules. As mentioned before, these are known in the euclidean setting and particular cases only, such as for $x$-independent compositions. We propose here a self-contained presentation of these results in our framework.

Our uniqueness theorem can be states as follows. For its proof, see Theorems \ref{un1}, \ref{smallT}.
\begin{thm}\label{unique} Suppose that \eqref{I}, \eqref{H1}-\eqref{H5} and \eqref{F1}-\eqref{F3} hold. Then \eqref{fmfg} admits a unique solution in the following cases:
\begin{itemize}
\item[(a)] The monotone case. If $H$ is convex and the following monotonicity condition holds
\begin{equation*}
\label{condunique}
\int_{\T^d}(F[m_1](x)-F[m_2](x))d(m_1-m_2)(x)>0\ ,\forall m_1,m_2\in\mathcal{P}(\T^d)\ ,m_1\neq m_2\ ,
\end{equation*}
then \eqref{fmfg} admits a unique solution.
\item[(b)] Small-time uniqueness. For $s\in(\frac12,1)$, there exists $T^*>0$, depending on $d,s,H,F,m_0, u_T$ such that for all $T\in(0,T^*]$, \eqref{fmfg} has at most a solution. 
\end{itemize}
\end{thm}

Finally, we mention that while this work was under preparation we discovered that E. R. Jakobsen and O. Ersland were currently studying systems similar to \eqref{fmfg}. 
The main difference with respect to this work are the assumptions on $s$ and $H$. In \cite{JE}, $s$ has to be greater than $1/2$, and $H$ is not necessarily convex but requires at most linear growth with respect to $Du$ in some cases. Since without convexity of $H$ one cannot rely on semiconcavity arguments, a different method to obtain crucial Lipschitz estimates is used. In \cite{JE} some models with local couplings are also analyzed. We stress that here we develop some function space techniques to study various regimes of regularity in the whole interval $s \in (0, 1)$.

\par\smallskip
\textit{Plan of the paper}. Section \ref{sintro} is devoted to some preliminary tools on the functional spaces used in the following sections. We prove the Sobolev embedding theorem for parabolic spaces in Subsection \ref{SecEmb}. Section \ref{SecFPHJB} is completely designated to the separate analysis of the viscous fractional Fokker-Planck and HJB equations. In particular, the existence result for the latter is given in Subsection \ref{SecExHJB}. In Section \ref{SecMFG} we prove both Theorem \ref{sigmapositivo} and Theorem \ref{vanishingviscosity}, postponing the uniqueness to Section \ref{SecUn}, where Theorem \ref{unique} is proven. As announced, in the appendices we gather regularity results in Sobolev and H\"older spaces for non-homogeneous fractional heat-type equations together with fractional Leibniz and composition rules on the torus.

\bigskip

{\bf Acknowledgements.} The authors are members of the Gruppo Nazionale per l'Analisi Matematica, la Probabilit\`a e le loro Applicazioni (GNAMPA) of the Istituto Nazionale di Alta Matematica (INdAM). This work has been partially supported by 
the Fondazione CaRiPaRo
Project ``Nonlinear Partial Differential Equations:
Asymptotic Problems and Mean-Field Games". The second-named author wishes to thank the Department of Mathematics of the University of Padova for the hospitality during the preparation of the paper.

\section{Fractional parabolic spaces}\label{sintro}
\subsection{H\"older spaces}
We first recall the definition of H\"older spaces on the torus and then define the classical parabolic H\"older spaces associated to the heat and fractional heat equation. Let $\alpha\in(0,1]$ and $k$ be a non-negative integer. A real-valued function $u$ defined on $\T^d$ belongs to $C^{k+\alpha}(\T^d)$ if $u\in C^k(\T^d)$ and 
\begin{equation*}
[D^{r}u]_{C^\alpha(\T^d)}:=\sup_{x \neq y\in\T^d}\frac{|D^ru(x)-D^ru(y)|}{\dist(x,y)^{\alpha}}<\infty
\end{equation*}
for each multi-index $r$ such that $|r|=k$, where $\dist(x,y)$ is the geodesic distance from $x$ to $y$ on $\T^d$. Note that in the definition of the previous (and following) seminorm, since $u$ can be seen as a periodic function on $\R^d$, $\dist(x,y)$ can be replaced by the euclidean distance $|x-y|$, and the supremum be taken in $\R^d$. We will denote by $\| \cdot \|_{\infty; \Omega}$ the sup-norm on 
$\Omega$ (and eventually drop $\Omega$ in the subscript if it is clear from the context).

Let now $I \subseteq [0,T]$ and $Q = \T^d \times I$. First define
\begin{equation*}
[u]_{C^{\alpha}_x(Q)}:=\sup_{t\in[0,T]}[u(\cdot,t)]_{C^{\alpha}(\T^d)}
\end{equation*}
and
\begin{equation*}
[u]_{C^{\beta}_t(Q)}:=\sup_{x\in\T^d}[u(x,\cdot)]_{C^{\beta}(I)}
\end{equation*}

For any integer $k$ we denote by $C^{2k,k}(Q)$ the set of functions $u = u(x,t):Q\rightarrow\R$ which are continuous in $Q$ together with all derivatives of the form $\partial^r_tD^{\beta}_xu$ for $2r+|\beta|\leq 2k$. Moreover, let $C^{2k+\alpha,k+\alpha/2}(Q)$ be functions of $C^{2k,k}(Q)$ such that the derivatives $\partial^r_tD^{\beta}_xu$, with $2r+|\beta| = 2k$, are $\alpha$-H\"older in $x$ and $\alpha/2$-H\"older in $t$, with norm
\[
\|u\|_{C^{2k+\alpha, k+\alpha/2}(Q)} = \sum_{2r+|\beta|\leq 2k} \|\partial^r_tD^{\beta}_xu\|_{\infty;Q} + \sum_{2r+|\beta| = 2k} [\partial^r_tD^{\beta}_xu]_{C^{\alpha}_x(Q)} + [\partial^r_tD^{\beta}_xu]_{C^{\alpha/2}_t(Q)}.
\]

For these classical parabolic H\"older spaces, we refer the interested reader to \cite{GMparabolic, KrylovbookHolder, LSU} for a more comprehensive discussion.

We now consider some more general H\"older spaces. 
Let $X$ be a Banach space and $\beta\in(0,1)$. Denote by $C^{\beta}(I; X)$ the space of functions $u: I \rightarrow X$ such that the norm defined as
\begin{equation*}
\norm{u}_{C^{\beta}(I; X)}:=\sup_{t\in I}\norm{u(t)}_X+\sup_{t \neq \tau}\frac{\norm{u(t)-u(\tau)}_X}{|t-\tau|^\beta}
\end{equation*}
is finite. Hence,  specializing to $X=C^{\alpha}(\T^d)$, $\alpha\in(0,1)$, we have that
 $C^{\beta}(I; C^{\alpha}(\T^d))$ is the set of functions $u: I \rightarrow C^{\alpha}(\T^d)$ with finite norm
\begin{equation*}
\norm{u}_{C^{\beta}(I;C^{\alpha}(\T^d))}:=\norm{u}_{\infty; Q}+\sup_{t\in I }[u(\cdot,t)]_{C^{\alpha}(\T^d)}
+[u]_{C^{\beta}(I;C^{\alpha}(\T^d))}\ ,
\end{equation*}
where the last seminorm is defined as
\begin{equation*}
[u]_{C^{\beta}(I;C^{\alpha}(\T^d))}:=\sup_{t \neq \tau\in I}\frac{\norm{u(\cdot,t)-u(\cdot,\tau)}_{C^{\alpha}(\T^d)}}{|t-\tau|^{\beta}}\ .
\end{equation*}
When dealing with regularity of parabolic equations driven by fractional diffusion, we also need the following H\"older spaces with different regularity in time and space. Following the lines of \cite{CF} and \cite{FRRO}, we define $\mathcal{C}^{\alpha,\beta}(Q)$ as the space of continuous functions $u$ such that the following H\"older parabolic seminorm is finite
\begin{equation}\label{fracholder}
[u]_{\mathcal{C}^{\alpha,\beta}(Q)}:= [u]_{C^{\alpha}_x(Q)}+[u]_{C^{\beta}_t(Q)}.
\end{equation}
The norm in the space $\mathcal{C}^{\alpha,\beta}(Q)$ is defined naturally as
\begin{equation*}
\norm{u}_{\mathcal{C}^{\alpha,\beta}(Q)}:=\|u\|_{\infty; Q}+[u]_{\mathcal{C}^{\alpha,\beta}(Q)}\ .
\end{equation*}
Note that if $\beta = \alpha/2$, the space $\mathcal{C}^{\alpha,\beta}(Q)$ coincides with $C^{\alpha, \alpha/2}(Q)$. As pointed out in \cite{FRRO}, the following equivalence between seminorms holds
\begin{equation*}
[u]_{\mathcal{C}^{\alpha,\beta}(Q)}\sim \sup_{x,y\in\T^d,t,\tau\in[0,T]}\frac{|u(x,t)-u(y,\tau)|}{\dist(x,y)^{\alpha}+|t-\tau|^{\beta}} \ .
\end{equation*}

\smallskip
All the spaces above can be defined analogously on $\R^d$ and $Q = \R^d \times I$. Moreover, if $u$ is a periodic function in the $x$-variable, norms on $\T^d$ and $\R^d$ coincide, e.g. $\|u\|_{C^\alpha(\T^d)} = \|u\|_{C^\alpha(\R^d)}$, ...

\begin{rem}\label{holderinc}
It is worth noticing that we have to distinguish the spaces $C^{\beta}([0,T];C^{\alpha}(\T^d))$ and $\mathcal{C}^{\alpha,\beta}(Q)$, since it results
\begin{equation*}
C^{\beta}([0,T];C^{\alpha}(\T^d))\subsetneq \mathcal{C}^{\alpha,\beta}(Q_T)\ .
\end{equation*}
It can be easily seen by taking $\beta=\alpha$ and a periodic function in the $x$-variable that behaves like $(x+t)^{\alpha}$ in a neighborhood of $(0,0)$ (see in particular \cite[Section 4]{Rabier}).
\end{rem}
\subsection{Fractional Sobolev and Bessel potential spaces}\label{susb}
Recall that $L^p(\T^d)$ is the space of all measurable and periodic functions belonging to $L_{\rm loc}^p(\R^d)$ with norm $\|\cdot\|_{p}=\|\cdot\|_{L^p((0,1)^d)}$. If $k$ is a non-negative integer, $W^{k, p}(\T^d)$ consists of $L^p(\T^d)$ functions with (distributional) derivatives in $L^p(\T^d)$ up to order $k$. For $\mu\in\R$ and $p\in(1,\infty)$, we can directly define the Bessel potential space $H_p^{\mu}(\T^d)$ as the space of all distributions $u$ such that $(I-\Delta)^{\frac{\mu}{2}}u\in L^p(\T^d)$, where $(I-\Delta)^{\frac{\mu}{2}}u$ is the operator defined in terms of Fourier series
\begin{equation*}
(I-\Delta)^{\frac{\mu}{2}} u(x)=\sum_{k\in\Z^d}(1+4\pi^2|k|^2)^{\frac{\mu}{2}} \hat u(k) e^{2\pi ik\cdot x}\ ,
\end{equation*}
where
\begin{equation*}
\hat u(k)=\int_{\T^d}u(x)e^{-2\pi ik\cdot x}dx\ .
\end{equation*}
The norm in $H_p^{\mu}(\T^d)$ will be denoted by 
\begin{equation*}
\norm{u}_{\mu,p}:=\norm{(I-\Delta)^{\frac{\mu}{2}}u}_{p}.
\end{equation*}
Note that $H_p^{k}(\T^d)$ coincides with $W^{k, p}(\T^d)$ when $k$ is a non-negative integer and $p\in(1,\infty)$, by standard arguments in Fourier series (see Remark \ref{isometry} below). Moreover, $C^{\infty}(\T^d)$ is dense in $H_p^{\mu}(\T^d)$, by a convolution procedure: this fact will be useful to prove several properties of Bessel spaces, as it is sufficient to argue in the smooth setting to get general results. 

Bessel potential spaces can be also constructed via complex interpolation. We will briefly present such a construction, that will be helpful to derive some useful properties of  $H_p^{\mu}(\T^d)$. For additional details, we refer to \cite[Chapter 2]{LunardiSNS}, \cite[Chapter 4]{BL} and \cite[Section 1.9]{trbookinterpolation}. In general, in complex interpolation theory one considers two  Banach spaces $X,Y$, that are continuously embedded in a Hausdorff topological vector space $Z$. Let $S$ be the set
\[
S:=\{z\in\mathbb{C}:0<\text{Re} z<1\}\ .
\]
We define
\begin{multline*}
\mathcal{H}_{X,Y}(S):=\{u(\theta) \, | \, u(\theta):\overline{S}\to X+Y\text{ bounded and continuous},\\
\text{holomorphic on $S$}, \|u(it)\|_{X},\|u(1+it)\|_{Y}\text{ bounded for $t\in\R$}\}
\end{multline*}
and we equip it with the norm
\[
\|u\|_{\mathcal{H}_{X,Y}(S)} = \max\{\sup_{t\in\R}\|u(it)\|_{X}, \sup_{t\in\R}\|u(1+it)\|_{Y}\}.
\]
For every $\theta\in[0,1]$ we define the complex interpolation space with respect to $(X,Y)$ as
\[
[X,Y]_{\theta}=\{u(\theta):u\in \mathcal{H}_{X,Y}(S)\}
\]
endowed with the norm
\[
\|f\|_{[X,Y]_{\theta}}:=\inf_{u\in\mathcal{H}_{X,Y}(S), u(\theta)=f}\|u\|_{\mathcal{H}_{X,Y}(S)}\ .
\]
Then, one has that $H_p^{\mu}(\T^d)$ can be obtained by complex interpolation between $L^p(\T^d)$ and $W^{k,p}(\T^d)$, see, e.g., \cite[Section 3]{ScT} or \cite[Theorem 6.4.5 and p. 170]{BL}, that is
\[
H^\mu_p(\T^d) \simeq [L^p(\T^d),W^{k,p}(\T^d)]_{\theta}, \qquad \text{where $\mu = k\theta$}.
\]

We briefly describe also some tools to construct real interpolation spaces, namely the so-called K-method and the trace method, referring, among others, to \cite[Chapter 1]{LunardiBook} or \cite[Chapter 1]{LunardiSNS} for additional details.
In general, real interpolation between $L^p(\T^d)$ and $W^{k,p}(\T^d)$ leads to spaces that do not coincide with Bessel potential spaces. Still, we will make use of this other class of fractional spaces to prove useful properties of $(-\Delta)^s$.
Let $X,Y$ be Banach spaces with $Y\subset X$, $\theta\in[0,1]$ and $p\in[1,\infty]$. For every $x\in X$ and $t>0$, define
\[
K(t,x,X,Y)=\inf_{x=a+b,a\in X,b\in Y}\|a\|_X+t\|b\|_Y\ .
\]
If $I\subset(0,\infty)$, we denote by $L^p_{*}(I)$ the Lebesgue space $L^p(I,\frac{dt}{t})$ and $L^{\infty}_{*}(I)=L^{\infty}(I)$. We define the real interpolation space $(X,Y)_{\theta,p}$ between the Banach spaces $X,Y$ as
\[
(X,Y)_{\theta,p}=\{x\in X+Y:t\mapsto t^{-\theta}K(t,x,X,Y)\in L^p_{*}(0,+\infty)\}
\]
endowed with the norm
\[
\|x\|_{\theta,p}=\|t^{-\theta}K(t,x,X,Y)\|_{L^p_{*}(0,+\infty)}\ .
\]
It can be proved that this is a Banach space. We remark that such a construction turns out to be useful to prove H\"older regularity of the solution of the fractional heat equation in Theorem \ref{Regularity}. Another frequent characterization of real interpolation spaces is given by means of the trace method (see \cite[Section 1.8.1]{trbookinterpolation}, \cite[Section 1.2.2]{LunardiBook} and \cite{LM}). Let $X,Y$ be Banach spaces as above. For $\alpha,p\in\R$ with $p\in(1,+\infty)$ satisfying $0<\alpha+\frac1p<1$, we define the space 
\[
W(p,\alpha,Y,X)=\{f:\R^+\to X: t^{\alpha}f(t)\in L^p(0,+\infty;Y)\text{ and }t^{\alpha}f'(t)\in L^p(0,+\infty;X)\}\ .
\]
It is a Banach space endowed with the norm
\[
\|f\|_{W(p,\alpha,Y,X)}:=\max\{\|t^{\alpha}f(t)\|_{L^p(0,+\infty;Y)}, \, \|t^{\alpha}f'(t)\|_{L^p(0,+\infty;X)}\}\ .
\]
We then identify with $T(p,\alpha,Y,X)$ the space of traces $u$ of those functions $f(t)\in W(p,\alpha,Y,X)$, equipped with the norm
\[
\|u\|_{T(p,\alpha,Y,X)}=\inf_{u=f(0)}\|f\|_{W(p,\alpha,Y,X)}
\]
By \cite[Proposition 1.2.10]{LunardiBook}, this provides a characterization for the real interpolation space $(X,Y)_{\theta,p}$ as a trace space. For $p\in(1,\infty)$, $\theta \in (0,1)$ and $\theta=\frac1p+\alpha$, we define fractional Sobolev spaces $W^{1-\theta,p}(\T^d)$ by 
\[
W^{1-\theta,p}(\T^d)=T(p,\alpha,W^{1,p}(\T^d),L^p(\T^d)).
\]
For $\mu > 1$, $W^{\mu,p}(\T^d)$ is defined as the space of functions in $W^{\lfloor{\mu}\rfloor,p}(\T^d)$ with derivatives of order $\lfloor{\mu}\rfloor$ in $W^{\mu-\lfloor{\mu}\rfloor ,p}(\T^d)$, while for $\mu < 0$ it is defined by duality.
Note that $T(p,\alpha,Y,X)=T(p',-\alpha,X',Y')$ by \cite[Theorem 1.2]{LM}.

We finally mention that spaces $W^{\mu,p}(\T^d)$ defined above can be characterized using the Gagliardo seminorm on $\T^d$ by transposing classical arguments on $\R^d$ (see, e.g., \cite{LunardiSNS}).


{\bf Parabolic spaces.} We proceed with the definitions of some functional spaces involving time and space weak derivatives. Let $Q=\T^d\times I$ be as before. For any integer $k$ and $p\geq1$, we denote by $W^{2k,k}_p(Q)$ the space of functions $u$ such that $\partial_t^{r}D^{\beta}_xu\in L^p(Q)$ for any multi-index $\beta$ and $r$ such that $|\beta|+2r\leq  2k$ endowed with the norm
\begin{equation*}
\norm{u}_{W^{2k,k}_p(Q)}=\left(\iint_{Q}\sum_{|\beta|+2r\leq2k}|\partial_t^{r}D^{\beta}_x u|^pdxdt\right)^{\frac1p}.
\end{equation*}
We now define the fractional generalization of the above spaces. Let again $\mu\in\R$ and $p \in (1, \infty)$. We denote by $\mathbb{H}_{p}^{\mu}(Q):=L^p(0,T;H_p^{\mu}(\T^d))$ the space of measurable functions $u:(0,T)\rightarrow H_p^{\mu}(\T^d)$ endowed with the norm
\begin{equation*}
\norm{u}_{\mathbb{H}_p^{\mu}(Q)}:=\left(\int_0^T\norm{u(\cdot,t)}^p_{H^{\mu}_p(\T^d)}dt\right)^{\frac1p}\ .
\end{equation*}
We define the space $\H_p^{\mu}(Q) = \H_p^{\mu; s}(Q)$ as the space of functions $u\in \mathbb{H}_p^{\mu}(Q)$ with $\partial_tu\in(\mathbb{H}_{p'}^{2s-\mu}(Q))'$ equipped with the norm
\begin{equation*}
\norm{u}_{\mathcal{H}_p^{\mu}(Q)}:=\norm{u}_{\mathbb{H}_p^{\mu}(Q)}+\norm{\partial_tu}_{(\mathbb{H}_{p'}^{2s-\mu}(Q))'}\ .
\end{equation*}
We refer the reader to \cite{CL}. Note that the above definitions make sense also when $s=1$ (we will usually drop the superscript $s$ for brevity). Those are natural spaces in the standard parabolic setting: see \cite{KrylovBookSPDE} and \cite{CT}, \cite[Chapter 6]{BKRS} for properties in the case $s = 1$. Note that $(\mathbb{H}_{p'}^{2s-\mu}(Q))'$ coincides with $\mathbb{H}_{p}^{\mu-2s}(Q)$.

Moreover, all the aforementioned spaces can be defined analogously on $\R^d$ and $\R^d \times I$, mutatis mutandis. In particular, one has to consider $(I-\Delta)^{\frac{\mu}{2}}u$ as the operator acting on tempered distributions in terms of the Fourier transform $\mathcal{F}$:
\[
\mathcal{F} [(I-\Delta)^{\frac{\mu}{2}}u ] (\xi) = (1 + 4\pi^2|\xi|^2)^{\frac{\mu}{2}} \, \mathcal{F}u(\xi), \qquad \forall \xi \in \R^d.
\]

\subsection{The fractional Laplacian on the torus}
In this section we recall the definition of the fractional Laplacian on the flat torus. Let $u:\T^d\rightarrow\R$. 
The fractional Laplacian on the torus can be defined via the multiple Fourier series
\begin{equation*}
(-\Delta_{\T^d})^{{\mu}}u(x)=(2\pi)^{2\mu}\sum_{k\in\Z^d}|k|^{2 \mu}\hat u(k) e^{2\pi ik\cdot x}\ ,\qquad \mu>0\ .
\end{equation*}
With a slight abuse of notation, we will denote this operator by $(-\Delta)^s$. Indeed, generally speaking $(-\Delta_{\T^d})^s$ coincides with the standard fractional Laplacian on $\R^d$ acting on periodic functions. We refer the reader to \cite{rs, CC} for additional details, and to \cite{rs2} for transference properties from the torus to the euclidean space. Note that in our analysis we never make use of the integral representation formula for the fractional Laplacian on the torus. 

We present two standard results that will be useful in the sequel

\begin{lemma}\label{intparts}
For every smooth $f,g$, the following identity holds true for any $s\in(0,1)$
\begin{equation*}
\int_{\T^d}(-\Delta)^sf gdx= { \int_{\T^d}(-\Delta)^{s/2}f \, (-\Delta)^{s/2} gdx} = \int_{\T^d}f(-\Delta)^s gdx\ .
\end{equation*}
\end{lemma}
\begin{proof}
The functions $f$ and $g$ can be written by multiple Fourier series expansion
\begin{equation*}
f(x)=\sum_{\nu\in\Z^d}\hat f(\nu)e^{2\pi i\nu\cdot x}\text{ and }g(x)=\sum_{\mu\in\Z^d}\hat g(\mu) e^{2\pi i\mu\cdot x}.
\end{equation*}
Then
\begin{multline*}
\int_{\T^d}(-\Delta)^sfgdx=(2\pi)^{2s}\int_{\T^d}\sum_{\nu,\mu\in\Z^d}|\nu|^{2s}\hat f(\nu)e^{2\pi i\nu\cdot x}\hat g(\mu)e^{2\pi i\mu\cdot x}dx\\
=(2\pi)^{2s}\sum_{\nu,\mu\in\Z^d}|\nu|^{2s}\hat f(\nu) \hat g(\mu)\int_{\T^d}e^{2\pi i(\nu+\mu)\cdot x}dx=(2\pi)^{2s}\sum_{\nu+\mu=0}|\mu|^{s}|\nu|^s\hat f(\nu)\hat g(\mu)\int_{\T^d}e^{2\pi i(\nu+\mu)\cdot x}dx\\=(2\pi)^{2s}\sum_{\nu,\mu\in\Z^d}|\mu|^{s}|\nu|^s\hat f(\nu)\hat g(\mu)\int_{\T^d}e^{2\pi i(\nu+\mu)\cdot x}dx\\
=(2\pi)^{2s}\int_{\T^d}\!\!\sum_{\nu,\mu\in\Z^d}|\nu|^s\hat f(\nu)e^{i\nu\cdot x}|\mu|^{s}\hat g(\mu)e^{2\pi i\mu\cdot x}dx
=\int_{\T^d}(-\Delta)^{\frac{s}{2}}f(-\Delta)^{\frac{s}{2}}gdx\ ,
\end{multline*}
where we used that $\int_{\T^d}e^{2\pi i(\nu+\mu)\cdot x}dx=0$ if and only if $\mu+\nu\neq0$ and the fact that the Fourier series defining $f$ and $g$ converge absolutely.
\end{proof}

\begin{rem}\label{isometry}
We point out that the operator $(I-\Delta)^{\frac{\mu}{2}}$ maps isometrically $H_p^{\eta+\mu}(\T^d)$ to $H_p^{\eta}(\T^d)$ (and therefore spaces $\H_p^{\eta+\mu}$ to $\H_p^{\eta}$) for any $\eta,\mu\in\R$, 
Moreover, for $\mu > 0$, the operator $(-\Delta)^{\frac{\mu}{2}}$ is bounded from $H_p^{\eta+\mu}(\T^d)$ to $H_p^{\eta}(\T^d)$. Indeed, $T^{\mu}:=[(-\Delta)^{\frac{\mu}{2}}(I-\Delta)^{-\frac{\mu}{2}}]$, $\mu>0$ is bounded in $L^p(\R^d)$ (see \cite[p. 133]{Stein1}), so
\begin{equation}\label{DI-Drn}
\| (-\Delta)^\frac{\mu}{2} u\|_{L^p(\R^d)} \le C(s,p) \|u\|_{H^{\mu}_p(\R^d)}.
\end{equation}
In other words, $(2\pi)^\mu|\xi|^{\mu}(1+4\pi^2|\xi|^2)^{-\frac{\mu}{2}}$ defines a Fourier multiplier on $L^p(\R^d)$. Then, by the transference result \cite[Theorem VIII.3.8]{SW}, the periodized operator given by
\[
\tilde{T}^{\mu}u:=\sum_{k\in\Z^d}(2\pi)^\mu|k|^{\mu}(1+4\pi^2|k|^2)^{-\frac{\mu}{2}}\hat u(k)e^{2\pi ik\cdot x}
\]
is in turn bounded in $L^p(\T^d)$. It then follows
\begin{equation}\label{DI-Dt}
\| (-\Delta)^\frac{\mu}{2} u\|_{L^p(\T^d)}= \| \tilde{T}^{\mu}(I-\Delta)^{\frac{\mu}{2}}u\|_{L^p(\T^d)}\leq C\|(I-\Delta)^{\frac{\mu}{2}}u\|_{L^p(\T^d)}=C\|u\|_{H^{\mu}_p(\T^d)}\ ,
\end{equation}
so $(-\Delta)^{\frac{\mu}{2}}$ is bounded from $H_p^{\mu}(\T^d)$ to $L^p(\T^d)$. The general case follows by using the isometry $(I-\Delta)^{\frac{\eta}{2}}$.

Similarly, $(1+(2\pi)^\mu|\xi|^{\mu})/(1+4\pi^2|\xi|^2)^{\frac{\mu}{2}}$ and $(1+4\pi^2|\xi|^2)^{\frac{\mu}{2}}/(1+(2\pi)^\mu|\xi|^{\mu}) $ define Fourier multipliers on $L^p(\R^d)$ for $1 < p <\infty$, and by continuity they transfer to $L^p(\T^d)$. This proves the equivalence of norms $\|\cdot\|_{\mu,p}$ and $\|\cdot\|_{p} + \|(-\Delta)^{\frac{\mu}{2}}\cdot\|_{p}$.

By analogous arguments involving Fourier multipliers, one proves that $H_p^{k}(\T^d)$ coincides with $W^{k, p}(\T^d)$ when $k$ is a non-positive integer and $p\in(1,\infty)$. See \cite{GrafakosModern} for the euclidean case, that easily transfers to the periodic setting.

\end{rem}

The following interpolation estimates hold.

\begin{lemma}\label{LpestimatesT}
Let $u\in L^p(\T^d)$, $p\in(1,\infty)$.
\begin{itemize}
\item[(i)] If $s\in(0,\frac12)$ and $D u\in L^p(\T^d)$, then for every $\delta > 0$ there exists $C(\delta) > 0$ depending on $\delta, d,s,p$ such that
\begin{equation*}
\norm{(-\Delta)^su}_p \leq \delta \norm{Du}_p+C(\delta) \norm{u}_p.
\end{equation*}
\item[(ii)] If $s\in[\frac12,1)$ and $D^2 u\in L^p(\T^d)$, then for every $\delta > 0$ there exists $C(\delta) > 0$ depending on $\delta, d,s,p$ such that
\begin{equation*}
\norm{(-\Delta)^su}_p\leq \delta \norm{D^2u}_p+C(\delta)\norm{u}_p.
\end{equation*}
\end{itemize}
\end{lemma}
\begin{proof}
The proof follows by interpolation arguments. We prove only the case (i), the other being similar. Since $H_p^{2s}(\T^d)\simeq [L^p(\T^d),W^{1,p}(\T^d)]_{\theta}$, $\theta=2s$, by \eqref{DI-Dt} and Young's inequality we have
\begin{equation*}
\norm{(-\Delta)^su}_p\leq \norm{u}_{2s,p}\leq C\norm{u}_p^{1-\theta}\norm{u}_{1,p}^{\theta}\leq (1-2s)\left(\frac{C}{\eps}\right)^{\frac{1}{1-2s}}\norm{u}_p+2s\epsilon^{\frac{1}{2s}}\norm{u}_{1,p}
\end{equation*}
where $C=C(d,s,p)$. We then conclude (i) by setting $\delta:=2s\epsilon^{\frac{1}{2s}}$ and $C(\delta):=(1-2s)\left(\frac{C}{\eps}\right)^{\frac{1}{1-2s}} +2s\epsilon^{\frac{1}{2s}}$.

\end{proof}

\subsubsection{Embedding Theorems for $H^\mu_p$}\label{SecEmb}

We recall some classical continuous embeddings for (stationary) Bessel potential spaces $H_p^{\mu}(\T^d)$.
\begin{lemma}\label{inclstat}
\begin{itemize}
  \item[(i)] Let $\nu,\mu\in\R$ with $\nu\leq\mu$, then $H_p^{\mu}(\T^d)\hookrightarrow H_p^{\nu}(\T^d)$.
  \item[(ii)] If $p\mu>d$ and $\mu-d/p$ is not an integer, then $H_p^{\mu}(\T^d)\hookrightarrow C^{\mu-d/p}(\T^d)$.
  \item[(iii)] Let $\nu,\mu\in\R$ with $\nu\leq\mu$, $p,q\in(1,\infty)$ and
  \[
  \mu-\frac{d}{p}=\nu-\frac{d}{q}\ ,
  \]
  then $H_p^{\mu}(\T^d)\hookrightarrow H_q^{\nu}(\T^d)$.
\end{itemize}
\end{lemma}
\begin{proof}
Item (i)-(iii) are proven in \cite[Corollary 13.3.9]{KrylovbookSobolev}, \cite[Theorem 13.8.1]{KrylovbookSobolev} and \cite[Theorem 13.8.7]{KrylovbookSobolev} respectively for the whole space case. The transference to the periodic setting can be obtained as follows. 
Let $\chi \in C^\infty_0(\R^d)$ be a cutoff function such that $\chi \equiv 1$ on the unit cube $[0,1]^d$ and $0 \le \chi \le 1$.

Let now $u$ be smooth function on $\T^d$, namely a smooth periodic function on $\R^d$. Then, it is easy to check that the extension operator
\begin{equation}\label{exte}
u \mapsto \tilde u = \chi u \qquad \text{on $\R^d$}
\end{equation}
extends to a linear continuous operator $W^{k,p}(\T^d) \to W^{k,p}(\R^d)$, for all non-negative integers $k$ and $p \ge 1$. The spaces $H^\mu_p(\R^d)$ and $H^\mu_p(\T^d)$ can be both obtained via complex interpolation, that is for some {$\theta \in (0,1)$} and $k \ge \mu \ge 0$, $H^\mu_p(\T^d) \simeq [L^p(\T^d),W^{k,p}(\T^d)]_{\theta} $ and $H^\mu_p(\R^d) \simeq [L^p(\R^d),W^{k,p}(\R^d)]_{\theta}$. Moreover, they coincide with $W^{\mu,p}(\T^d)$ and $W^{\mu,p}(\R^d)$ respectively when $\mu$ is a non-negative integer. Therefore, the extension operator \eqref{exte} is also bounded on $H^\mu_p(\T^d) \to H^\mu_p(\R^d)$ by interpolation (see \cite[Theorem (a), p. 59]{trbookinterpolation} and \cite[Chapter 2]{LunardiSNS}).Thus, for all $\mu \ge 0$,
\begin{equation}\label{eeq}
\|u\|_{L^p(\T^d)} \le \|\tilde u\|_{L^p(\R^d)} \le C \|\tilde u\|_{H^\mu_p(\R^d)}  \le C_1 \|u\|_{H^\mu_p(\T^d)},
\end{equation}
that implies (i) in the case $\nu = 0$ (note that for the first inequality in \eqref{eeq} to be true, it is crucial to work in $L^p$, so that the restriction operator $L^p(\R^d) \to L^p(\T^d)$ is continuous). The general case $\nu \neq 0$ follows by applying to \eqref{eeq} the isometry $(I-\Delta)^{\nu/2}$. Items (ii) and (iii) are obtained analogously.
\end{proof}

\subsubsection{Embedding Theorems for parabolic spaces $\mathcal{H}_p^{\mu}$}

We now prove continuous embedding theorems for the spaces $\mathcal{H}_p^{\mu}(Q_T) = \mathcal{H}_p^{\mu; s}(Q_T)$, where $Q_T = \T^d \times (0, T)$. As usual, we will denote continuous embeddings of Banach spaces by the symbol $X \hookrightarrow Y$. All the results of this section are valid for $s \in (0, 1]$. We will basically follow the strategy of \cite[Theorem 7.2]{KrylovBookSPDE}, where analogous results are proven for (stochastic) spaces associated to heat-type equations (that is, for $s =1$) on $\R^d \times (0, T)$ (see also \cite{KrylovJFA}). In addition, we refer to \cite[Theorem 6.2.2]{BKRS}, \cite[Proposition 2.2]{CT} and \cite[Theorem A.3]{MPR} for the case $s, \mu=1$. We first state the main result of this section and, at the end, we will deduce some useful corollaries.

\begin{thm}\label{Embedding} Let $\varepsilon > 0$, $\mu \in \R$, $p > 1$, $u\in\mathcal{H}_p^{\mu}(Q_T)$ and $u(0) \in H^{\mu - 2s/p+\eps}_p(\T^d)$. If $\beta$ is such that 
\[ \frac s p < {\beta} < s,\]
then $u\in C^{\frac \beta s-\frac1p}([0,T];H_p^{\mu-2\beta}(\T^d))$. In particular, there exists $C > 0$ depending on $d,p,\beta,s,T,\eps$, such that 
\begin{equation*}
\norm{u(\cdot,t)-u(\cdot,\tau)}^p_{\mu-2\beta,p}\leq C |t-\tau|^{\frac{\beta}{s}p-1}(\norm{u}_{\mathcal{H}_p^{\mu}(Q_T)}+ \norm{u(0)}_{\mu-2s/p+\eps,p})\end{equation*}
for $0 \le t,\tau\leq T$. Hence,
\begin{equation}\label{holder-bessel}
\norm{u}_{C^{\frac{\beta}{s}-\frac1p}([0,T];H_p^{\mu-2\beta}(\T^d))}\leq C(\norm{u}_{\mathcal{H}_p^{\mu}(Q_T)}+\norm{u(0)}_{\mu-2s/p+\eps,p})\ .
\end{equation}
Note that the constant $C$ remains bounded for bounded values of $T$.
\end{thm}

We first need some estimates in the spaces of Bessel potentials for the semigroup $\mathcal{T}_t$ associated to the fractional Laplacian. Recall that for a given smooth $u$, $\mathcal{T}_t u := v(t)$, where $v$ solves
\[
\begin{cases}
\partial_tv + (-\Delta)^s v = 0 & \text{in $Q_T$}, \\
v(0) = u& \text{in $\T^d$}.
\end{cases}
\]
Then we have the following standard representation formula that can be obtained via Fourier transform
\begin{equation}\label{repre}
\mathcal{T}_tu(x)=\int_{\R^d}p_t(x-y)u(y)dy = p_t \star_{\R^d} u(x),
\end{equation}
where $p_t(x):= \mathcal{F}^{-1}(e^{-t|\xi|^{2s}})(x) = \int_{\R^d}e^{2\pi ix\cdot\xi}e^{-t|\xi|^{2s}}d\xi$. In the periodic case, namely  if $u:\T^d\to\R$, defining $\hat p_t(x):=\sum_{z\in\Z^d}p_t(x+z) =  \sum_{z\in\Z^d}e^{-t|z|^{2s}}e^{2\pi i z\cdot x}$, we note that
\begin{multline}\label{repre2}
\mathcal{T}_tu(x)=
\int_{\R^d}p_t(x-y)u(y)dy =
\sum_{z\in\Z^d}\int_{(0,1)^d+z} p_t(x-y)u(y+z) dy \\
=\int_{\T^d}\sum_{z\in\Z^d} p_t(x-y-z)u(y)dy=
\int_{\T^d}\hat p_t(x-y)u(y)dy = \hat p_t \star_{\T^d} u(x).
\end{multline}
This shows that some properties of the fractional heat semigroup on the whole space $\R^d$ can be directly transferred to the periodic case. First, $\norm{p_t}_{L^1(\R^d)}=1$, so $\norm{\hat p_t}_{L^1(\T^d)}=1$, readily yielding
\begin{equation}\label{claim2}
\norm{\mathcal{T}_tf}_{p}\leq \norm{f}_{p} \qquad \forall p \in [1, \infty]\ ,
\end{equation}
by Young's inequality for convolutions. Moreover, $p_t(x) = t^{-d/2s} p_1(t^{-1/2s}x)$ by rescaling, hence for a multiindex $\beta$ we have
\begin{equation}\label{TT2}
\norm{D^{\beta}\hat p_t}_{L^1(\T^d)} \leq \norm{D^{\beta} p_t}_{L^1(\R^d)} \leq t^{-|\beta|/2s} \norm{D^{\beta} p_1}_{L^1(\R^d)} \leq C t^{-|\beta|/2s},
\end{equation}
by boundedness of $\norm{D^{\beta} p_1}_{L^1(\R^d)}$ (see, e.g.,  \cite[Lemma 2.4]{Z}).
\begin{rem}\label{DerivateKernel}
Representation formula \eqref{repre2} and decay estimates \eqref{claim2} imply that for any $f\in C^{\infty}(\T^d)$ and multiindices $k,m\in\N$,
\begin{equation}\label{TT3}
\|D^{k+m}\mathcal{T}_tf\|_{p}\leq Ct^{-\frac{k}{2s}}\|D^mf\|_{p} \qquad \forall p \in [1,\infty]\ .
\end{equation}
On the one hand, this shows that for $t > 0$, $\mathcal{T}_t$ maps $C^{m}(\T^d)$ onto $C^{k+m}(\T^d)$. On the other hand, exploiting the density of $C^{\infty}(\T^d)$ in $H_p^{\mu}(\T^d)$, one obtains that $\mathcal{T}_t$ is bounded from $W^{m,p}(\T^d)$ to $W^{k+m,p}(\T^d)$.
\end{rem}
In addition, note that, for $\mu\in\R$, it results
\begin{equation}\label{exchange}
\mathcal{T}_t(I-\Delta)^{\frac{\mu}{2}}u=(I-\Delta)^{\frac{\mu}{2}}\mathcal{T}_tu.
\end{equation}
The equality can be verified by taking its Fourier transform.

\begin{lemma}\label{decaytorus}
\begin{itemize}
 \item [\it{(i)}] For any $p>1$ and $\nu\in\R$,$\gamma\geq0$, we have for all $f\in H_p^{\nu}(\T^d)$
\begin{equation*}
\norm{\mathcal{T}_tf}_{\nu+\gamma,p}\leq Ct^{-\gamma/2s}\norm{f}_{\nu,p}\ ,
\end{equation*}
where $C=C(\nu,\gamma,d,s,p)$.
  \item[(ii)]  For any $\theta\in[0,s]$ and $p>1$, there exists a constant $C=C(d,s,p,\theta)$ such that, for all $f\in H_p^{2 \theta}(\T^d)$, it holds
\begin{equation}\label{i2T}
\norm{\mathcal{T}_tf-f}_{p}\leq Ct^{\theta/s}\norm{f}_{2\theta,p}.
\end{equation}
\end{itemize}
\end{lemma}
\begin{proof}
To prove (i) one can restrict without loss of generality to $\nu = 0$, since the general case will follow by replacing $f$ by $(I-\Delta)^{-\nu}f$. The proof is a consequence of (complex) interpolation between inequalities \eqref{claim2} and \eqref{TT3}, see e.g. \cite[Theorem (a) p. 59]{trbookinterpolation}.

We prove (ii), and follow the strategy of \cite[Lemma 7.3]{KrylovBookSPDE}. 
First, by (i) with $\nu=2\theta$ and $\gamma=2s-2\theta \ge 0$, we get
\begin{equation}\label{auxstima}
\norm{\mathcal{T}_tf}_{2s,p}\leq Ct^{\frac{\theta}{s}-1}\norm{f}_{2\theta,p}\ ,
\end{equation}
where $C=C(d,p,\theta,s)$. 
Note that $(\mathcal{T}_tf)'=-(-\Delta)^s\mathcal{T}_tf$. Hence, we have
\begin{equation*}
\norm{(\mathcal{T}_t-1)f}_p\leq\int_0^t\norm{[(-\Delta)^s(I-\Delta)^{-s}](I-\Delta)^{s}\mathcal{T}_{\tau}f}_p d\tau
\end{equation*}
\begin{equation*}
 \leq C\int_0^t\norm{\mathcal{T}_{\tau}f}_{2s,p} d\tau\leq C\norm{f}_{2\theta,p} \int_0^t\tau^{\frac{\theta}{s}-1}d\tau=Ct^{\frac{\theta}{s}}\norm{f}_{2\theta,p}
\end{equation*}
where we used \eqref{auxstima} and the fact that $[(-\Delta)^s(I-\Delta)^{-s}]$ is bounded in $L^p(\T^d)$ (see Remark \ref{isometry}).
\end{proof}

\begin{rem}\label{anse} We observe that $-(-\Delta)^s$ generates an analytic semigroup $\mathcal{T}_t$ on $L^p(\T^d)$ for all $p > 1$, since the following inequality
\[
\norm{-(-\Delta)^s\mathcal{T}_tf}_{p}\leq Ct^{-1}\norm{f}_{p}
\]
holds (then, argue via \cite[Theorem 2.5.2]{Pazy} for example). The above estimate is in turn a straightforward consequence of Lemma \ref{decaytorus}-(i) with $\nu=0$ and $\gamma=2s$. 
\end{rem}

We recall the following useful lemma, and refer to \cite[Lemma 7.4]{KrylovBookSPDE} (and references therein) for its proof.
\begin{lemma}\label{lemmakrylov1}
Let $p\geq1$ and $\alpha p>1$. Then, for any continuous $L^p$-valued function $h(\cdot)$ and $\tau\leq t$ we have
\begin{equation}\label{i3}
\norm{h(t)-h(\tau)}_{p}^p\leq C(\alpha,p)(t-\tau)^{\alpha p-1}\int_{\tau}^t\int_{\tau}^t\mathbf{1}_{r_2>r_1}\frac{\norm{h(r_2)-h(r_1)}_{p}^p}{|r_2-r_1|^{1+\alpha p}}dr_1dr_2
\end{equation}
\begin{equation*}
= C(\alpha,p)(t-\tau)^{\alpha p-1}\int_0^{t-\tau}\frac{d\gamma}{\gamma^{1+\alpha p}}\int_{\tau}^{t-\gamma}\norm{h(r+\gamma)-h(r)}_p^pdr\ .
\end{equation*}
As a consequence one has
\begin{equation}\label{i4}
\sup_{0\leq \tau<t\leq T}\frac{\norm{h(t)-h(\tau)}_p^p}{(t-\tau)^{\alpha p-1}}\leq C(\alpha,p)\int_0^T\int_0^T\mathbf{1}_{r_2>r_1}\frac{\norm{h(r_2)-h(r_1)}_{p}^p}{|r_2-r_1|^{1+\alpha p}}dr_1dr_2\ ,
\end{equation}
where $\mathbf{1}_{A}$ denotes the indicator function of a given set $A$.
\end{lemma}

We now proceed with the proof of the embeddings of $\mathcal{H}_p^{\mu}$.

\begin{proof}[Proof of Theorem \ref{Embedding}] Note first that since the operator $(I-\Delta)^{\frac{\eta}{2}}$ maps isometrically $\mathcal{H}_p^{\mu}(Q_T)$ onto $\mathcal{H}_p^{\mu-\eta}(Q_T)$ for any $\eta,\mu$ (see Remark \ref{isometry}), we just consider the case $2\beta = \mu$. We than have to prove that
\begin{equation*}
\norm{u(t)-u(\tau)}^p_{p}\leq C |t-\tau|^{\frac{\beta}{s}p-1}(\norm{u}_{\mathcal{H}_p^{2\beta}(Q_T)}+\norm{u(0)}_{2\beta-2s/p+\eps,p}),
\end{equation*}
for $0 \le t,\tau\leq T$.

Define
\begin{equation}\label{eqv}
f:=\partial_tu+(-\Delta)^su,
\end{equation}
and by Duhamel's formula we have
\begin{equation*}
u(t)=\mathcal{T}_tu(0)+\int_{0}^t\mathcal{T}_{t-\tau}f(\tau)d\tau\ ,
\end{equation*}
where $\mathcal{T}_t$ is defined at the beginning of this section. We claim that
\begin{equation*}
u(r+\gamma)-u(r)=(\mathcal{T}_{\gamma}-1)u(r)+\int_0^{\gamma}\mathcal{T}_{\gamma-\rho}f(r+\rho)d\rho\ .
\end{equation*}
Indeed we have
\begin{multline*}
\mathcal{T}_{\gamma}u(r)-u(r)+\int_0^{\gamma}\mathcal{T}_{\gamma-\rho}f(r+\rho)d\rho \\
=\mathcal{T}_{r+\gamma}u(0)+\int_0^r \mathcal{T}_{r+\gamma-\tau}f(\tau)d\tau-u(r)+\int_0^{\gamma}\mathcal{T}_{\gamma-\rho}f(r+\rho)d\rho\\
=\mathcal{T}_{r+\gamma}u(0)+\int_{0}^{r+\gamma}\mathcal{T}_{r+\gamma-\tau}f(\tau)d\tau-u(r)=u(r+\gamma)-u(r)\ .
\end{multline*}
Therefore,
\begin{equation*}
\norm{u(r+\gamma)-u(r)}_p^p\leq C(A(r,\gamma)+B(r,\gamma))\ ,
\end{equation*}
where 
\begin{equation*}
A(r,\gamma)=\norm{(\mathcal{T}_{\gamma}-1)\, u(r)}_p^p
\end{equation*}
and
\begin{equation*}
B(r,\gamma)=\norm{\int_0^{\gamma}\mathcal{T}_{\gamma-\rho}f(r+\rho)d\rho}_p^p = \norm{\int_0^{\gamma}\mathcal{T}_{\omega}f(r+\gamma-\omega)d\omega}_p^p.
\end{equation*}

Choose  $\alpha$ so that $\frac1p < \alpha < \frac{\beta}{s}$. By Lemma \ref{lemmakrylov1} we have
\begin{equation}\label{eest1}
\norm{u(t)-u(\tau)}_p^p\leq C(\alpha,p)(t-\tau)^{\alpha p-1}(I(t,\tau)+J(t,\tau))\ ,
\end{equation}
where
\begin{equation*}
I(t,\tau)=\int_{0}^{t-\tau}\frac{d\gamma}{\gamma^{1+\alpha p}}\int_{\tau}^{t-\gamma}A(r,\gamma)dr
\end{equation*}
and
\begin{equation*}
J(t,\tau)=\int_{0}^{t-\tau}\frac{d\gamma}{\gamma^{1+\alpha p}}\int_{\tau}^{t-\gamma}B(r,\gamma)dr\ .
\end{equation*}
To estimate $B$, we use H\"older's inequality and Lemma \ref{decaytorus}-(i) (with $\nu=0$ and $\gamma=2s-2\beta\in(0,1)$. We have
\begin{multline*}
B(r,\gamma)=\int_{\T^d}\left|\int_0^{\gamma}\omega^{\frac{\beta}{s}-1}\omega^{1-\frac{\beta}{s}}\mathcal{T}_\omega f(r+\gamma-\omega)d\omega\right|^p dx \\
\leq \left(\int_0^{\gamma}\omega^{(\frac{\beta}{s}-1)q}d\omega\right)^{\frac{p}{q}}\int_0^{\gamma}\omega^{(1-\frac{\beta}{s})p}\int_{\T^d}|\mathcal{T}_{\omega}f(r+\gamma-\omega)|^pdx \, d\omega \\
 \leq C(d,p,\beta,s)\gamma^{\frac{\beta}{s}p-1}\int_0^{\gamma}\norm{f(r+\gamma-\omega)}_{2\beta-2s,p}^pd\omega \\
= C(d,p,\beta,s)\gamma^{\frac{\beta}{s}p-1}\int_0^{\gamma}\norm{f(r+\rho)}_{2\beta-2s,p}^pd\rho.
\end{multline*}
This and the inequality $\alpha<\frac{\beta}{s}$ give
\begin{multline}\label{eest2}
J(t,\tau)\leq C(d,p,\alpha,\beta,s)\int_0^{t-\tau}\frac{d\gamma}{\gamma^{2+(\alpha-\frac{\beta}{s})p}}\int_0^{\gamma}d\rho\int_{\tau}^{t-\gamma}\norm{f(r+\rho)}_{2 \beta-2s,p}^pdr
\\
\leq C(d,p,\alpha,\beta,s)\int_0^{t-\tau}\frac{d\gamma}{\gamma^{2+(\alpha-\frac{\beta}{s})p}}\int_0^{\gamma}d\rho\int_0^t\norm{f(r)}_{2\beta-2s,p}^pdr \\
= C(d,p,\alpha,\beta,s)(t-\tau)^{(-\alpha+\frac{\beta}{s})p}\int_0^t\norm{f(r)}_{2 \beta-2s,p}^pdr.
\end{multline}
Recalling that $f=\partial_tu+(-\Delta)^su$, by \eqref{DI-Drn} 
\begin{multline*}
J(t,\tau) \leq C(d,p,\alpha,\beta,s)(t-\tau)^{(-\alpha+\frac{\beta}{s})p}\int_0^t\left(\norm{\partial_tu(r)}^p_{2\beta-2s,p}+\norm{(-\Delta)^su(r)}^p_{2\beta-2s,p}\right)dr \\ \leq
C(d,p,\alpha,\beta,s)(t-\tau)^{(-\alpha+\frac{\beta}{s})p}\int_0^t\left(\norm{\partial_tu(r)}^p_{2\beta-2s,p}+\norm{u(r)}^p_{2\beta,p}\right)dr\\
= C(d,p,\alpha,\beta,s)(t-\tau)^{(-\alpha+\frac{\beta}{s})p}\left(\norm{\partial_tu}^p_{\mathbb{H}_p^{2\beta-2s}(Q_T)}+\norm{u}^p_{\mathbb{H}_p^{2\beta}(Q_T)}\right)\\
= C(d,p,\alpha,\beta,s)(t-\tau)^{(-\alpha+\frac{\beta}{s})p}\norm{u}^p_{\mathcal{H}_p^{2\beta}(Q_T)}.
\end{multline*}
 
To estimate $I$, we apply Lemma \ref{decaytorus}-(ii) with $\theta=\beta\in(0,s)$ and Theorem \ref{fracregtor} to get
\begin{multline*}
\int_0^t A(r,\gamma) dr \leq C(d,p,\beta,s) \gamma^{\frac{\beta}{s}p} \int_0^t \norm{u(r)}^p_{2\beta,p} d r \\ \leq C \gamma^{\frac{\beta}{s}p}  \|u\|_{\mathbb{H}_p^{2\beta}(Q_T)}^p \le
C_1(d,p,\alpha,\beta,s,T,\eps) \gamma^{\frac{\beta}{s}p}  (\norm{f}^p_{\mathbb{H}_p^{2\beta-2s}(Q_T)}+\norm{u(0)}^p_{2\beta-2s/p+\eps,p}).
\end{multline*}
Thus,
\begin{multline*}
I(t,\tau)\leq \int_0^{t-\tau}\frac{d\gamma}{\gamma^{1+\alpha p}}\int_0^t A(r,\gamma)dr \\
\leq C(d,p,\alpha,\beta,s,T,\eps)(t-\tau)^{(\frac{\beta}{s}-\alpha)p} (\norm{\partial_tu+(-\Delta)^su}^p_{\mathbb{H}_p^{2\beta-2s}(Q_T)}+\norm{u(0)}^p_{2\beta-2s/p+\eps,p}) \\
\leq C(d,p,\alpha,\beta,s,T,\eps)(t-\tau)^{(-\alpha+\frac{\beta}{s})p}(\norm{u}^p_{\mathcal{H}_p^{2\beta}(Q_T)}+\norm{u(0)}^p_{2\beta-2s/p+\eps,p}).
\end{multline*}
Finally, combining the last inequality with \eqref{eest1} and \eqref{eest2}, we proved that
\begin{equation}\label{quasiemb}
\norm{u(t)-u(\tau)}^p_{p}\leq C(d,p,\alpha,\beta,s,T,\eps)|t-\tau|^{\frac{\beta}{s}p-1}(\norm{u}^p_{\mathcal{H}_p^{2\beta}(Q_T)}+\norm{u(0)}^p_{2\beta-2s/p+\eps,p})\ .
\end{equation}
To obtain \eqref{holder-bessel}, in the special case $\mu = 2\beta$, it remains to show that 
\begin{equation}\label{hb}
\sup_{t\leq T}\norm{u(t)}_{p}\leq C(\norm{u}_{\H_p^{2\beta}(Q_T)}+\norm{u(0)}_{2\beta-2s/p+\eps,p})\ ,
\end{equation}
This is a consequence of \eqref{quasiemb} and the continuous embedding of $H_p^{2\beta-2s/p+\eps}(\T^d)$ into $L^p(\T^d)$, as $\beta > 2s/p$. Indeed,
\[
\norm{u(t)}^p_{p}\leq C(\beta, s, p, d) \norm{u(0)}^p_{2\beta-2s/p+\eps, p} + C T^{\frac{\beta}{s}p-1}(\norm{u}^p_{\mathcal{H}_p^{2\beta}(Q_T)}+\norm{u(0)}^p_{2\beta-2s/p+\eps,p}).
\]
\end{proof}

We now present some continuous embedding results that stem from Thereom \ref{Embedding}.

\begin{prop}\label{embe2} Let $\eps > 0$, $q\geq p>1$, $0 \le \theta \le 1$ and $\mu, \eta \in \R$ be such that
\begin{equation}\label{etamu}
\eta<\mu+\frac{d}{q}-\frac{d+2s(1-\theta)}{p}.
\end{equation}
Then, for any $u\in\mathcal{H}_p^{\mu}(Q_T)$,
\begin{equation*}
\left(\int_0^T\norm{u(\cdot,t)}_{\eta,q}^{\frac{p}{\theta}}dt\right)^{\theta}\leq C(\norm{u}_{\mathcal{H}_p^{\mu}(Q_T)}^p+\norm{u(0)}^p_{\mu-2s/p+\eps,p})\ .
\end{equation*}
In particular, if $\mu > 0$, $1 < p < \frac{d + 2s}{\mu}$ and $\frac 1 q > \frac 1 p - \frac{\mu}{d + 2s}$,
\[
\|u\|_{L^q(Q_T)} \leq C(\norm{u}_{\mathcal{H}_p^{\mu}(Q_T)}+\norm{u(0)}_{\mu-2s/p+\eps,p})
\]
Here, $C$ depends on $d,p,q,\mu,\eta,\theta,T,s$, but remains bounded for bounded values of $T$.
\end{prop}

\begin{proof}
Let $0 < \beta < s$ to be chosen. Recall that, for any $\theta \in [0, 1]$, if $\nu = \nu(\beta)=(\mu-2\beta)(1-\theta)+\theta \mu$, then $H^\nu_p$ can be obtained by (complex) interpolation between $H^\mu_p$ and $H^{\mu- 2\beta}_p$ (see, e.g., \cite[Theorem 6.4.5]{BL}). Moreover, $H^\nu_p$ is continuously embedded in $H^{\nu+d/q-d/p}_q$ in view of Lemma \ref{inclstat}. Hence, for a.e. $t$,
\begin{equation*}
c(d, p, s, \beta)\norm{u(t)}_{\nu-\frac{d}{p}+\frac{d}{q},q}\leq\norm{u(t)}_{\nu,p}\leq \norm{u(t)}_{\mu- 2\beta,p}^{1-\theta}\norm{u(t)}_{\mu,p}^{\theta}.
\end{equation*}

By \eqref{etamu}, we can choose ${2\beta} > \frac {2s} p$ so that $\eta \le \nu(\beta)-\frac{d}{p}+\frac{d}{q} < \mu+\frac{d}{q}-\frac{d+2s(1-\theta)}{p}$, and therefore
\begin{multline*}
\left(\int_0^T\norm{u(t)}_{\eta,q}^{\frac{p}{\theta}}dt\right)^{\theta}\leq
C\left(\int_0^T\norm{u(t)}_{\mu-2\beta,p}^{(1-\theta)\frac{p}{\theta}}\norm{u(t)}_{\mu,p}^pdt\right)^{\theta} \\
 \leq C\sup_{t\leq T}\norm{u(t)}_{\mu-2\beta,p}^{(1-\theta)p}\left(\int_0^T\norm{u(t)}_{\mu,p}^pdt\right)^{\theta} \\
\leq C(\norm{u}_{\mathcal{H}_p^{\mu}(Q_T)} + \norm{u(0)}_{\mu-2s/p+\eps,p})^{(1-\theta)p}\norm{u}_{\mathbb{H}_p^{\mu}(Q_T)}^{\theta p} \\ \leq C(\norm{u}_{\mathcal{H}_p^{\mu}(Q_T)}+ \norm{u(0)}_{\mu-2s/p+\eps,p})^{p}
\end{multline*}
where, in the last inequality, we used Theorem \ref{Embedding} and Young's inequality.

The last statement follows by choosing $\eta = 0$ and $\theta = p/q$.
\end{proof}

\begin{prop}\label{embe3}
 Let $\eps > 0$, $\frac 1 2 < s < 1$, $p>\frac{d+2s}{2s-1}$ and $u(0) \in H^{\mu - 2s/p+\eps}_p(\T^d)$. Then for all $ u\in\mathcal{H}_p^{2s-1}(Q_T)$ the following inequality holds
\[
\norm{u}_{\mathcal{C}^{\gamma,\frac{\gamma}{2s}}(Q_T)} \le C(\norm{u}_{\mathcal{H}_p^{2s-1}(Q_T)}+\norm{u(0)}_{2s-1-2s/p+\eps,p}),
\]
where 
\[
\gamma = s - \frac s p - \frac d {2p} - \frac 1 2,
\]
and $C$ depends on $d, s, p, T$.
\end{prop}
\begin{proof}
First apply Theorem \ref{Embedding} with $\mu = 2s-1$ to get
\begin{equation*}
\mathcal{H}_p^{\mu}(Q_T)\hookrightarrow C^{\frac \beta s-\frac1p}([0,T]; H_p^{\mu-2\beta}(\T^d))\ .
\end{equation*}
Then, exploit the embedding $H_p^{\mu-2\beta}(\T^d)\hookrightarrow C^{\mu-2\beta-\frac{d}{p}}(\T^d)$ of Lemma \ref{inclstat}. By choosing $\beta$ so that $\frac \beta s-\frac1p = \frac{\gamma}{2s}$ and $\gamma$ as in the statement, then $\mu-2\beta-\frac{d}{p} = \gamma$, and one concludes by the inclusion of $C^{\frac{\gamma}{2s}}(C^\gamma)$ into ${\mathcal{C}^{\gamma,\frac{\gamma}{2s}}}$ (see Remark \ref{holderinc}).
\end{proof}

\begin{rem}
We point out that all the estimates carried out in this section can be proven exactly in the same manner for the $\R^d$ case. Indeed, the arguments turn around decay estimates for the fractional heat operator and fractional heat parabolic regularity that hold to the same extent on $\R^d$ and $\T^d$.
\end{rem}

\subsection{Relation between $H^\mu_p$ and $W^{\mu,p}$}
We prove the embeddings between $W^{\mu,p}$ and $H_p^{\mu}$ via the trace method. Without going into the details, we mention that when $p=2$ the space $W^{\mu,2}$ coincides with $H_2^{\mu}$ by properties of Fourier transform. For general $p \neq 2$, we follow the lines of \cite[Theorem 3.1]{LM}, shortening their proof by using decay estimates of Lemma \ref{decaytorus}.
\begin{lemma}\label{inclusioni}
For every $\eps>0$, $\mu\in\R$ and $1<p<\infty$ we have
\[
H_p^{\mu+\eps}(\T^d)\hookrightarrow W^{\mu,p}(\T^d)\hookrightarrow H_p^{\mu-\eps}(\T^d)\ .
\]
\end{lemma}
\begin{proof}
\textit{Step 1}. We first prove that $H_p^{1-\theta+\eps}(\T^d)\hookrightarrow W^{1-\theta,p}(\T^d)$ for every $\eps>0$ and $\theta\in(0,1)$. To show this, it is sufficient to confine ourselves to the case $\eps<\theta$ since $H_p^{\nu}(\T^d)\hookrightarrow H_p^{\eta}(\T^d)$ for every $\nu,\eta\in\R$ such that $\nu>\eta$. Set $\lambda:=1-\theta+\eps$ and take $u\in H_p^{\lambda}(\T^d)$. We need to show the existence of $f(t)$ such that
\[
t^{\alpha}f(t)\in L^p(0,1;W^{1,p}(\T^d))
\]
\[
t^{\alpha}f'(t)\in L^p(0,1;L^p(\T^d))
\]
and
\[
f(0)=u
\]
are fulfilled, for $\alpha = \theta - 1/p$. Once one finds such $f(t)$, it is sufficient to multiply it by a continuously differentiable function $\zeta(t)$ for $t\in[0,+\infty)$, which vanishes for $t\geq1$ and it is identically 1 for $t\in[0,1/2]$ and then set $g(t)=\zeta(t)f(t)$ for $t\in[0,1]$ and $g(t)=0$ for $t>1$. As a consequence, it follows that $t^{\alpha}g(t)\in L^p(0,+\infty;W^{1,p}(\T^d))$, $t^{\alpha}g'(t)\in L^p(0,+\infty;L^p(\T^d))$ and $g(0)=f(0)=u\in W^{1-\theta,p}(\T^d)$. To reach our goal, we use the solution of the fractional heat equation with $s = 1/2$ and initial data equal to $u$, that is
\[
f(t):=\mathcal{T}_tu\ ,
\]
where here $\mathcal{T}_t$ is the semigroup associated to the half-laplacian. It is clear that $f(0)=u$. We show only that $t^{\alpha}f(t)\in L^p(0,1;W^{1,p}(\T^d))$, the other case being similar. By Lemma \ref{decaytorus}-(i) with $\nu=\lambda$ and $\gamma=\theta-\epsilon>0$ we have
\begin{multline*}
\left(\int_0^1\|t^{\alpha}\mathcal{T}_tu\|_{1,p}^pdt\right)^{\frac1p}\leq C_1\left(\int_0^1t^{\alpha p}t^{-(\theta-\eps)p}\|u\|_{\lambda,p}^pdt\right)^{\frac1p} \\
\leq C_2\left(\int_0^1t^{(\alpha-\theta+\eps)p}dt\right)^{\frac1p}\|u\|_{\lambda,p}\leq C_3.
\end{multline*}
\par\smallskip
\textit{Step 2}. We claim that for every $\eps>0$ it results $W^{1-\theta,p}(\T^d)\hookrightarrow H_p^{1-\theta-\eps}(\T^d)$. By isometry (see Remark \ref{isometry}), the operator $(I-\Delta)^{\frac12}$ maps $W^{1,p}(\T^d)$ onto $L^p(\T^d)$ and $L^p(\T^d)$ onto $W^{-1,p}(\T^d)$. In addition, it also maps $H^{1-\theta+\epsilon}(\T^d)$ onto $H^{-\theta+\epsilon}(\T^d)$. By definition we have that it is also an isometry between
\[
T(p,\alpha,W^{1,p}(\T^d),L^p(\T^d))=W^{1-\theta,p}(\T^d)
\] 
and
\begin{multline*}
T(p,\alpha,L^p(\T^d),W^{-1,p}(\T^d))=(T(p',-\alpha,W^{1,p'}(\T^d),L^{p'}(\T^d)))'\\ =(W^{1-(1/p'-\alpha),p'}(\T^d))'=W^{-\theta,p}(\T^d).
\end{multline*}
By Step 1 we obtain
\[
H^{-\theta+\eps}_p(\T^d)\hookrightarrow W^{-\theta,p}(\T^d) \ ,
\]
which turns out to hold for every $\eps>0$. By duality we also conclude $W^{\theta,p'}(\T^d)\hookrightarrow H_{p'}^{\theta-\eps}(\T^d)$ and hence the validity of the claim after replacing $\theta$ by $1-\theta$.
\par\smallskip
\textit{Step 3}. Suppose $\mu\geq0$. We first prove the left inclusion $H_p^{\mu+\epsilon}(\T^d)\hookrightarrow W^{\mu,p}(\T^d)$. Let $u\in H_p^{\mu+\epsilon}(\T^d)$. Then $D^ku\in L^p(\T^d)$ for all $|k|\leq[\mu]$, where $[\cdot]$ stands for the integer part. On the other hand, $D^ku\in H_p^{\mu+\eps-[\mu]}(\T^d)$ for $k=[\mu]$, which gives by Step 1 $D^ku\in W^{\mu-[\mu],p}(\T^d)$. Then $u\in W^{\mu,p}(\T^d)$. Conversely, if $u\in W^{\mu,p}(\T^d)$, it means that $u\in H_p^{[\mu]}(\T^d)$. Thus in view of Step 2 we obtain $D^ku\in H_p^{\mu-[\mu]-\eps}(\T^d)$, namely $u\in H_p^{\mu-\eps}(\T^d)$ which in turn implies $W^{\mu,p}(\T^d)\hookrightarrow H_p^{\mu-\eps}(\T^d)$. The case $\mu<0$ follows by the previous one arguing by duality.
\end{proof}

\section{Fractional Fokker-Planck and HJB equations}\label{SecFPHJB}

\subsection{On the fractional Fokker-Planck equation}\label{SecFP}
In this section we gather some results on fractional Fokker-Planck equations in the periodic setting of the form 
\begin{equation}\label{FFP}
\begin{cases}
\partial_tm-\sigma \Delta m +(-\Delta)^s m+\dive(bm)=0 & \text{ in }\T^d\times(0,T)\ ,\\
m(x,0)=m_0(x)& \text{ in }\T^d\ ,
\end{cases}
\end{equation}
with $\sigma \ge 0$ and $m_0\in L^\infty(\T^d)$. When $\sigma = 0$, we expect low regularity of solutions, in particular when $0< s < 1/2$. In this case we will adopt the usual notion of weak solution, with the following integrability requirements.
\begin{defn}\label{defFP}
Let $b\in L^\infty(Q_T)$ be such that\footnote{In what follows, we will denote by $[u]^-$ the negative part of $u$.} $[{\rm div}\,b]^- \in L^\infty(Q_T)$. A function \begin{equation}\label{integrm}
m\in  L^2(0,T;H_2^{s}(\T^d)) = \mathbb{H}_2^{s}(Q_T) \qquad \text{with} \quad \partial_tm \in L^2(0,T;H_2^{-1}(\T^d)) = \mathbb{H}_2^{-1}(Q_T)
\end{equation}
is a weak solution to \eqref{FFP} if, for every $\varphi\in C^{\infty}(\T^d\times[0,T))$, one has
\begin{equation*}
\iint_{Q_T}-m\partial_t\varphi-bm\cdot D\varphi +(-\Delta)^{\frac{s}{2}}m(-\Delta)^{\frac{s}{2}}\varphi \,dxdt=\int_{\T^d}\varphi(x,0)m_0(x)\,dx\ .
\end{equation*}
\end{defn}
\begin{rem}\label{uniqueFP} It can be verified that \eqref{integrm} implies $m \in C([0,T]; H^{(s-1)/{2}}_2(\T^d))$, see e.g. \cite[p. 480]{DL}. This suggests, by a density argument, that test functions $\varphi$ in the previous formulation can be chosen so that $\varphi \in L^2(0,T;H_2^{1}(\T^d))$ with $ \partial_t\varphi \in L^2(0,T;H_2^{-s}(\T^d))$, therefore satisfying $\varphi \in C([0,T]; H_2^{{(1-s)}/{2}}(\T^d))$. In this case the integration by parts in time formula holds (with an abuse of notation, integration in space is hiding duality pairings here):
\[
\iint_{Q_T}\varphi \partial_t m + m \partial_t\varphi  \, dxdt = \int_{\T^d}\varphi(x,T)m(x,T)\,dx -  \int_{\T^d}\varphi(x,0)m(x,0)\,dx.
\]

We also point out that solutions defined as in Definition \ref{defFP} are unique, as a consequence of the crucial unilateral bound on ${\rm div}\,b$. This can be justified formally by multiplying the equation by $m$ itself and deriving an usual $L^2$-energy estimate (as in \eqref{eneE} below). Since $m$ itself cannot be a test function because of the ``asymmetric'' integrability requirements on $m$ and $\partial_tm$, one has to perform a preliminary regularization procedure via convolution (see, e.g., \cite{Bris2008, WeiTian} and \cite{TesiAle}).
\end{rem}

We will need the following estimates independent of $\sigma$, for classical solutions of the viscous problem.

\begin{prop}\label{stabFP}
 Let $\sigma \ge 0$, $m_0 \in C(\T^d)$ and $b \in C_x^1(Q_T)$ such that
\[
\|m_0\|_{\infty} + \|b\|_{\infty}+ \|[{\rm div }\,b]^-\|_{\infty} \le K.
\]
Then, there exists $C= C(K)$ such that for every classical solution $m$ to \eqref{FFP} it holds
\begin{align}
&\|m\|_{\infty;Q_T} \le C, \label{minfty} \\
 \sigma \iint_{Q_T} |Dm|^2 \, dxdt+ & \iint_{Q_T} [(-\Delta)^{s/2}m]^2\, dxdt \le C, \label{gradients} \\
&  \|\partial_tm\|_{\mathbb{H}_2^{-1}(Q_T)} \le C.
\end{align}
\end{prop}
\begin{proof}
By standard comparison arguments involving the function
\[
w(x,t):=m(x,t)e^{-(K+\eps)t}-\|m_0\|_{\infty}
\]
with $\eps \to 0$ (see e.g.  \cite[Section II.2]{GMparabolic}), one concludes
\[
\|m\|_{\infty;Q_T}\leq \|m_0\|_{\infty}e^{KT}.
\]
Multiply the equation in \eqref{FFP} by $m$ 
and integrate over $Q_T$ to get 
\[
\frac12\int_0^T\frac{d}{dt}\norm{m}_{L^2(\T^d)}^2 - \sigma\iint_{Q_T}m\Delta m\,dxdt  + \iint_{Q_T}m(-\Delta)^sm\,dxdt =  - \iint_{Q_T}m\dive(bm)dxdt
\]
Using Lemma \ref{intparts} and integrating by parts we have
\begin{multline}\label{eneE}
\frac12\int_0^T\frac{d}{dt}\norm{m(\cdot,t)}_{L^2(\T^d)}^2 + \sigma\iint_{Q_T}|Dm|^2dxdt + \iint_{Q_T}[(-\Delta)^{\frac{s}{2}}m]^2\,dxdt = \iint_{Q_T}mb\cdot Dm\, dxdt \\ = -\frac12\iint_{Q_T}(\dive b)m^2\,dxdt.
\end{multline}
Using that $[\dive(b)]^{-}\leq K$ and the $L^\infty$ bound on $m$ (one could also argue via Gronwall's lemma), we obtain
\[
\frac12\norm{m(T)}_{L^2(\T^d)}^2+\sigma\iint_{Q_T}|Dm|^2dx+\iint_{Q_T} [(-\Delta)^{s/2}m]^2\, dx\leq C(K) + \frac12\norm{m(0)}_{L^2(\T^d)}^2
\]
which gives the desired inequality \eqref{gradients}.

The last estimate follows by observing that, using the equation in \eqref{FFP},
\begin{multline*}
\left|\iint_{Q_T}\partial_t m \varphi \, dxdt\right| \le \|b\|_{L^{\infty}(Q_T)} \|m\|_{L^2(Q_T)} \| D\varphi \|_{L^2(Q_T)} + \|(-\Delta)^{\frac{s}{2}}m\|_{L^2(Q_T)} \|\varphi\|_{\mathbb{H}_2^{s}(Q_T)} \\ \le C \|\varphi\|_{\mathbb{H}_2^{1}(Q_T)}.
\end{multline*}
\end{proof}

\subsection{On the fractional HJB equation}
\subsubsection{Semiconcavity estimates}
This subsection is devoted to the analysis of semiconcavity properties of solutions to backward fractional HJB equations
\begin{equation}\label{xfracHJsemicon}
\begin{cases}
-\partial_tu-\sigma\Delta u+(-\Delta)^su+H(x,Du)=V(x,t)&\text{ in }Q_T\ ,\\
u(x, T)=u_T(x)&\text{ in }\T^d\ ,
\end{cases}
\end{equation}
We prove in particular that $u$ is semiconcave, with semiconcavity constant depending on the data and independent of $\sigma$. First, we stress that when $\sigma=0$ we mean that $u$ is a weak (energy) solution according to the following
\begin{defn}\label{weak}
Let $\sigma = 0$ and $V$ be a continuous function on $Q_T$. We say that $u \in \H_2^s(Q_T)$ with $Du \in L^{\infty}(Q_T)$ is a weak solution to \eqref{xfracHJsemicon} if
\begin{multline*}
-\int_{\T^d}\varphi(x,T)u_T(x)dx+\iint_{Q_T} \partial_t\varphi u dxdt + \iint_{Q_T}(-\Delta)^{\frac{s}{2}}u(-\Delta)^{\frac{s}{2}}\varphi\, dxdt +\iint_{Q_T}H(x, Du) \varphi dxdt \\
=  \iint_{Q_T} V \varphi \,dxdt
\end{multline*}
for all $\varphi \in C^{\infty} (\T^d\times(0,T])$.
\end{defn}

\begin{rem}\label{u0l2}
We make a preliminary observation, which we will use in the sequel. Recall that $u\in \H_2^s(Q_T)$ means $u\in L^2(0,T;H_2^s(\T^d))$ with $\partial_t u\in L^2(0,T;H_2^{-s}(\T^d))$. Note that $\H_2^s(Q_T)$ is continuously embedded into $C(0, T; L^2(\T^d))$ in view of \cite[Theorem XVIII.2.1]{DL}), so this is equivalent to
\[
\iint_{Q_T} [-\partial_t u \varphi + (-\Delta)^{\frac{s}{2}}u(-\Delta)^{\frac{s}{2}}\varphi + H(x, Du) \varphi] dxdt =  \iint_{Q_T} V \varphi \,dxdt
\]
for all $\varphi \in  \H_2^s(Q_T)$, and $u(T) = u_T$ in the $L^2$-sense. Uniqueness of solutions in this sense holds by usual energy arguments (see also Remark \ref{uniqueFP}), and is based on the crucial property $Du \in L^{\infty}(Q_T)$ and the $C^1$ regularity of $H$.
\end{rem}

\begin{prop}\label{semiconcav}
Assume that $V \in C^{2+\alpha,1+\alpha/2}(Q_T)$, \eqref{H1} and \eqref{H3}-\eqref{H5} hold, and
\[
\norm{V}_{C^{2}_x(Q_T)} + \|u_T\|_{C^2(\T^d)} \leq K
\]
for some $K > 0$. Then every classical solution $u$ to \eqref{xfracHJsemicon} satisfies
\[
D^2 u(x,t) \le C \, I \qquad \text{on $Q_T$},
\]
where $C$ depends on $K$.
\end{prop}

The proof will be accomplished via the so-called adjoint method, that is, by using information of the dual linearized problem. This procedure is particularly effective when the Hamiltonian lacks uniform convexity. Here, we are inspired by some results in \cite{GomesBook}, see also references therein. We stress that we do not require convexity of $H$, but just assumptions (H1) and (H3)-(H5). Generally, for uniformly convex Hamiltonians similar results can be obtained in a more straightforward way through maximum principle arguments. When dealing with non-convex Hamiltonians, such approach fails in general. 

For any given $\rho_{\tau}\in C^{\infty}(\T^d)$, $\rho_{\tau}\geq0$, $\tau\in[0,T)$ and $\norm{\rho_{\tau}}_{L^1(\T^d)}=1$ we consider the adjoint equation
\begin{equation}\label{adjo}
\begin{cases}
\partial_t\rho-\sigma\Delta\rho+(-\Delta)^s\rho-\dive(D_pH(x,Du)\rho)=0&\text{ in }\T^d\times[\tau,T]\ ,\\
\rho(x,\tau)=\rho_{\tau}(x)&\text{ on }\T^d\ .
\end{cases}
\end{equation}

We have the following preliminary result
\begin{lemma}\label{reprsemic} There exists a classical solution $\rho$ to \eqref{adjo}. Moreover,
\[
\int_{\tau}^T\int_{\T^d}|Du|^{\gamma}\rho \, dxdt \le C,
\]
where $C$ depends on $K$ and not on $\rho_\tau$ nor $\tau$.
\end{lemma}

\begin{proof}
The well-posedness of \eqref{adjo} is a consequence of \cite[Theorem II.3.1]{GMparabolic} and the regularity assumptions on $H$ and $u$.
By multiplying the fractional HJB equation by $\rho$ and the adjoint equation by $u$, one easily obtains the following formula
\begin{equation}\label{reprformula}
\int_{\T^d}u(x,\tau)\rho_{\tau}(x)dx=\int_{\T^d}u(x,T)\rho(x,T)dx +\int_\tau^{T}\int_{\T^d}V\rho \, dxdt+
\end{equation}
\begin{equation*}
+\int_{\tau}^{T}\int_{\T^d}(D_pH(x,Du)\cdot Du-H(x,Du))\rho \, dxdt.
\end{equation*}
Then, by (H1) we get
\begin{multline}\label{ineq1}
\int_{\T^d}u(x,\tau)\rho_\tau(x)dx\geq \int_{\tau}^{T}\int_{\T^d}V\rho\,dxdt +C_H\int_{\tau}^T\int_{\T^d}|Du|^{\gamma}\rho\,dxdt-\\-c_H\int_{\tau}^T\int_{\T^d}\rho\,dxdt
+\int_{\T^d}\rho(x,T)u(x,T)dx.
\end{multline}
Then, since $u$ is a classical solution to \eqref{xfracHJsemicon}, a standard linearization argument and the application of the Comparison Principle for linear viscous integro-differential PDE (see, e.g. \cite[Section II.2]{GMparabolic}) yield
\begin{equation}\label{comp}
\norm{u}_{\infty;Q}\leq \norm{u_T}_{\infty;\T^d} + T\big(\norm{V}_{\infty;Q} + \norm{H(\cdot,0)}_{\infty;\T^d}\big).
\end{equation}
Finally, plugging \eqref{comp} in \eqref{ineq1} and using the fact that $\|\rho(t)\|_1 = 1$ for all $t$, we conclude the desired estimate.
\end{proof}


We now prove the semiconcavity estimate.
\begin{proof}[Proof of Proposition \ref{semiconcav}]
Since $V \in C^{2+\alpha,1+\alpha/2}(Q_T)$, by a bootstrap argument $u$ belongs to $C^{4+\alpha,2+\alpha/2}(Q_T)$ (see Proposition \ref{existenceHJ} below). So, we can differentiate twice the equation in any direction $\xi\in\R^d$, $|\xi|=1$. Observe that $v=u_{\xi}$ satisfies
\begin{equation*}
-\partial_tv-\sigma\Delta v+(-\Delta)^sv+D_pH(x,Du)\cdot Dv+D_{\xi}H(x,Du)=V_{\xi}\ ,\quad v(x,0)=u_{\xi}(0)
\end{equation*}
and $w=u_{\xi\xi}$ solves
\begin{equation}\label{eqw}
-\partial_tw-\sigma\Delta w+(-\Delta)^sw+Dv\cdot D^2_{pp}H(x,Du)Dv+D_pH(x,Du)\cdot Dw+
\end{equation}
\begin{equation*}
+2D_{p\xi}^2H(x,Du)\cdot Dv+D_{\xi\xi}^2H(x,Du)=V_{\xi\xi}\ ,\quad w(x,0)=u_{\xi\xi}(0)\ .
\end{equation*}
Then, multiply \eqref{eqw} by the adjoint variable $\rho$ satisfying \eqref{adjo} and integrate over $\T^d\times[\tau,T]$ to get
\begin{equation*}
\int_{\T^d}w(x,\tau)\rho_{\tau}(x)\,dx+\int_{\tau}^T\int_{\T^d}Dv\cdot D^2_{pp}H(x,Du)Dv\rho\, dxdt=\int_{\T^d}w(x,T)\rho(x,T)\,dx-
\end{equation*}
\begin{equation*}
-2\int_{\tau}^T\int_{\T^d}D_{p\xi}^2H(x,Du)\cdot Dv\rho\, dxdt-\int_{\tau}^T\int_{\T^d}D_{\xi\xi}^2H(x,Du)\rho\, dxdt+\int_{\tau}^T\int_{\T^d}V_{\xi\xi}\rho\, dxdt\ .
\end{equation*}
On one hand, by (H5) we have
\begin{equation*}
\int_{\tau}^T\int_{\T^d}Dv\cdot D^2_{pp}H(x,Du)Dv\rho\, dxdt\geq C_1\int_{\tau}^T\int_{\T^d}|Du|^{\gamma-2}|Dv|^2\rho\,dxdt-\tilde{C}_1\int_{\tau}^T\int_{\T^d}\rho\,dxdt
\end{equation*}
and hence, using also (H3)-(H4), we conclude
\begin{multline*}
\int_{\T^d}w(x,\tau)\rho_{\tau}(x)dx+C_1\int_{\tau}^T\int_{\T^d}|Du|^{\gamma-2}|Dv|^2 \rho \, dxdt-\tilde{C}_1\int_{\tau}^T\int_{\T^d}\rho\,dxdt  \\ \leq \int_{\T^d}w(x,T)\rho(x,T)dx+C_2\int_{\tau}^T\int_{\T^d}|Du|^{\gamma-1}|Dv|\rho dxdt+C_3\int_{\tau}^T\int_{\T^d}|Du|^{\gamma}\rho\, dxdt \\ +(\tilde{C}_2+\tilde{C}_3)\int_{\tau}^T\int_{\T^d}\rho\,dxdt+\int_{\tau}^T\int_{\T^d}V_{\xi\xi}\rho\, dxdt.
\end{multline*}
Now, we apply Young's inequality to the second term on the right-hand side of the above inequality to get
\begin{equation*}
\int_{\tau}^T\int_{\T^d}|Du|^{\gamma-1}|Dv|\rho\,dxdt\leq \frac{\epsilon^2}{2}\int_{\tau}^T\int_{\T^d}|Du|^{\gamma-2}|Dv|^2\rho dxdt+\frac{1}{\epsilon^2}\int_{\tau}^T\int_{\T^d}|Du|^{\gamma}\rho dxdt.
\end{equation*}
Taking $\epsilon$ so that $C_1=\frac{\epsilon^2}{2}$ we finally obtain the estimate
\begin{multline*}
\int_{\T^d}w(x,\tau)\rho_{\tau}(x)dx\leq  \int_{\T^d}w(x,T)\rho(x,T)dx+\left(\frac{1}{2C_1}+C_3\right)\int_{\tau}^T\int_{\T^d}|Du|^{\gamma}\rho\, dxdt+\\
+\int_{\tau}^T\int_{\T^d}V_{\xi\xi}\rho\,dxdt+\tilde{C}_4\ .
\end{multline*}
During the above computations $C_i=C_i(C_H)$. By Lemma \ref{reprsemic} we finally deduce the desired semiconcavity estimate after passing to the supremum over $\rho_{\tau}$.
\end{proof}
\begin{rem}\label{fracdiff}
The viscosity parameter $\sigma$ does not play any role in the above proof, and hence if $u$ is sufficiently regular to perform a differentiation procedure in the classical sense, the above scheme can be carried out with merely fractional diffusion of any order 
 $s\in(0,1)$. 
\end{rem}

We now turn to space-time H\"older bounds for (forward) fractional HJB equations with bounded right hand side. These will be useful in the vanishing viscosity limit to have uniform convergence of solutions, and therefore to bring to the limit the viscosity notion.
\begin{prop}\label{hc}
Let $f\in L^{\infty}(Q_T)$ and $u$ be a classical solution to
\[
\begin{cases}
\partial_tu-\sigma\Delta u+(-\Delta)^su=f(x,t) & \text{ in }Q_T \\
u(x,0)=u_0(x)&\text{ in }\T^d.
\end{cases}
\]
with $u_0\in C^1(\T^d)$. Then
\begin{equation}\label{holder}
\|u\|_{\mathcal{C}^{\alpha,\beta}(Q_T)}\leq C
\end{equation}
for some $\alpha,\beta\in(0,1)$, where the constant $C$ depends only on $\|f\|_{L^\infty(Q_T)}, \|u_0\|_{C^1(\T^d)}$ and is independent of $\sigma$.
\end{prop}
\begin{rem}
To prove the above result, we need to show the counterpart of Lemma \ref{decaytorus} for the semigroup $\bar{\mathcal{T}}_t$ generated by the full operator $\sigma\Delta-(-\Delta)^s$. We point out that the two semigroups $e^{-t(-\Delta)^s}$ and $e^{t\sigma\Delta}$ commute, and therefore
\[
\bar{\mathcal{T}}_t=e^{-t(-\Delta)^s}(e^{t\sigma\Delta}).
\]
\end{rem}
\begin{proof}[Proof of Proposition \ref{hc}]
We observe that by Lemma \ref{decaytorus}-(i) and \eqref{claim2}, it is straightforward to see that, for $\nu\in\R$, $p > 1$ and $\gamma\geq0$ we have
\begin{equation}\label{bobo}
\|\bar{\mathcal{T}}_tf\|_{\nu+\gamma,p}\leq Ct^{-\gamma/2s}\norm{g}_{\nu,p}\ .
\end{equation}
Note that $C$ does not depend on $\sigma$ here.

Write $u$ using Duhamel's formula, that is $u(t) = u_1(t) + u_2(t)$, where
\begin{equation*}
u_1(t)=\bar {\mathcal{T}}_t u_0, \qquad u_2(t) = \int_{0}^t\bar{\mathcal{T}}_{t-\tau}f(\tau)d\tau.
\end{equation*}
The estimate of $u_1(t):=\bar{\mathcal{T}}_tu_0$ follows using the same argument as in Theorem \ref{fracregtor} and the estimates in Lemma \ref{decaytorus}. We focus on $u_2(t)=\int_0^t\bar{\mathcal{T}}_{t-\tau}f(\tau)d\tau$. Take $\nu = 0$, $\gamma = \frac{s}{p}$ in \eqref{bobo} to get
\[
\|\bar{\mathcal{T}}_{t-\tau}f\|_{s,p}^p\leq C(t-\tau)^{-1/2}\|f\|_{L^\infty(Q_T)}^p\ .
\]
Therefore
\[
\|u_2\|_{\mathbb{H}_p^s(Q_T)}= \left(\int_0^T\norm{u_2(t)}_{s,p}^p\right)^{\frac1p}\leq CT^{\frac{3}{2p}}\|f\|_{L^{\infty}(Q_T)}
\]
Since $u_2$ solves $\partial_t u_2+(-\Delta)^s u_2=f$, one has
\[
\int_0^T\norm{\partial_tu_2(t)}^p_{-s,p}dt\leq C_1\left(\int_0^T\|(-\Delta)^su_2\|_{-s,p}^p+\|f\|_{-s,p}^pdt\right)\leq C_2\|f\|_{L^{\infty}(Q_T)}\ ,
\]
yielding the full estimate
\[
\norm{u}_{\H_p^{s}(Q_T)}\leq C(\norm{f}_{L^\infty(Q_T)}+\norm{u_0}_{s-2s/p+\epsilon})
\]
for $\eps<\frac{2s}{p}$. Then, for $p>\frac{d+2s}{s}$, by Sobolev embedding theorems in Proposition \ref{Embedding} we conclude
\[
\norm{u}_{\mathcal{C}^{\alpha,\beta}(Q_T)} \leq C\norm{u}_{\H_p^{s}(Q_T)} \le C_1\ . 
\]
\end{proof}
\subsubsection{Existence of solutions}\label{SecExHJB}
In this section we prove an existence result for backward integro-differential HJB equations of the form
\begin{equation}\label{HJ}
\begin{cases}
-\partial_tu-\Delta u+(-\Delta)^su+ H(x,Du)=V(x,t) &\text{ on }Q_T\ ,\\
u(x,T)=u_T(x) &\text{ on }\T^d\ .
\end{cases}
\end{equation}

\begin{prop}\label{existenceHJ}
Let $V\in C^{2+\alpha,1+\alpha/2}(Q_T)$, $H$ satisfying \eqref{H1}-\eqref{H5} and $u_T\in C^{4+\alpha}(\T^d)$. Then, there exists a unique solution $u\in C^{4+\alpha,2+\alpha/2}(Q_T)$ to \eqref{HJ}, and the following estimate holds
\begin{equation}\label{Lib}
\norm{u}_{C^{4+\alpha,2+\alpha/2}(Q_T)}\leq C(\norm{V}_{C^{2+\alpha,1+\alpha/2}(Q_T)}+\norm{u_T}_{C^{4+\alpha}(\T^d)})\ .
\end{equation}
\end{prop}

The crucial step to obtain this existence result are the semiconcavity estimates of the previous section, that yield a priori gradient bounds of solutions. Then, the construction of a solution follows by standard arguments. Since we were not able to find a similar result in the literature, we detail the proof here for the convenience of the reader.

\begin{proof}
Step 1: Local existence on $Q_\tau = \T^d \times (T-\tau, T)$ . Let $\tau\leq1$ and
\[
\mathcal{S}_a:=\left\{u\in W^{2,1}_p(Q_\tau):u(T)=u_T\ ,\|u\|_{W^{2,1}_p(Q_\tau)}\leq a\ ,p>d+2\right\}
\]
be the space on which we apply the contraction mapping principle. The parameter $a$ will be chosen large enough. Fix $z\in W^{2,1}_p(Q_\tau)$, $p>d+2$ and let $w = Jz$ be the solution of the problem
\begin{equation}\label{frozen}
\begin{cases}
-\partial_tw-\Delta w=V-H(x,Dz)-(-\Delta)^sz&\text{ in }\T^d\times(T-\tau,T]\ ,\\
w(x,T)=u_T(x)&\text{ in }\T^d.
\end{cases}
\end{equation}
By standard (local) parabolic regularity theory (see \cite[Theorem IV.9.1]{LSU} or \cite{CGM}), since the right hand side of the equation in \eqref{frozen} is in $L^p(Q_\tau)$, \eqref{frozen} admits a unique solution $w\in W^{2,1}_p(Q_{\tau})$ satisfying the following estimate
\[
\|w\|_{W^{2,1}_p(Q_\tau)}\leq C(\|V\|_{L^p(Q_\tau)}+\|H(x,Dz)\|_{L^p(Q_\tau)}+\|(-\Delta)^sz\|_{L^p(Q_\tau)}+\|u_T\|_{W^{2-2/p,p}(\T^d)}).
\]
We show that we can choose $\tau\in(0,T]$ sufficiently small so that $\|Jz\|_{W^{2,1}_p(Q_\tau)}\leq a$. By \cite[Lemma 2.4]{CGM}
\[
\|H(x,Dz)\|_{L^p(Q_\tau)}\leq C_1\tau^{\frac{1}{2p}}\|H(x,Dz)\|_{L^{2p}(Q_{\tau})}\leq C_2\tau^{\frac{1}{2p}}\|Dz\|^{\gamma}_{\infty;Q_{\tau}}
\]
Moreover, by \cite[Proposition 2.5]{CGM} we have
\[
\|Dz\|_{\infty;Q_\tau}\leq C_3(\|z\|_{W^{2,1}_p(Q_\tau)}+\|u_T\|_{W^{2-2/p,p}(\T^d)})\ ,
\]
which gives
\[
\|H(x,Dz)\|_{L^p(Q_\tau)}\leq C_4\tau^{\frac{1}{2p}}(\|z\|_{W^{2,1}_p(Q_\tau)}^\gamma +\|u_T\|_{W^{2-2/p,p}(\T^d)}^\gamma)\ .
\]
Concerning the fractional term we observe that if either $s\in(0,\frac12)$ or $s\in\left[\frac12,1\right)$, then by Lemma \ref{LpestimatesT} we get for some $\delta>0$
\begin{equation*}
\norm{(-\Delta)^sz}_{L^p(Q_\tau)}\leq \delta\norm{z}_{W^{2,1}_p(Q_\tau)}+C(\delta)\norm{z}_{L^p(Q_\tau)}
\end{equation*}
where $C(\delta)>0$ grows as $\delta$ approaches to 0. Then, note that by writing
\[
z(\cdot,s)=u_T(\cdot)-\int_s^T\partial_{t}z(\cdot,\omega)d\omega\ ,
\]
we obtain
\[
\|z\|_{L^p(Q_{\tau})}\leq \tau^{\frac1p}\|u_T\|_{L^p(\T^d)}+\tau\|\partial_tz\|_{L^p(Q_\tau)}\ .
\]
Then
\begin{multline*}
\|w\|_{W^{2,1}_p(Q_\tau)}\leq C\left[\max\{\|z\|_{W^{2,1}_p(Q_\tau)},\|z\|^{\gamma}_{W^{2,1}_p(Q_\tau)}\}(\tau^{\frac{1}{2p}}+C(\delta)\tau+\delta)\right.\\
\left.+(\tau^{\frac1p}+\tau^{\frac{1}{2p}})\max\{\|u_T\|_{L^p(\T^d)},\|u_T\|^{\gamma}_{W^{2-2/p,p}(\T^d)}\}+\|V\|_{L^p(Q_\tau)}\right] \\
\leq C\left[\max\{\|z\|_{W^{2,1}_p(Q_\tau)},\|z\|^{\gamma}_{W^{2,1}_p(Q_\tau)}\}(\tau^{\frac{1}{2p}}(1+C(\delta))+\delta)\right.\\
\left.+2\tau^{\frac{1}{2p}}\max\{\|u_T\|_{L^p(\T^d)},\|u_T\|^{\gamma}_{W^{2-2/p,p}(\T^d)}\}+\|V\|_{L^p(Q_\tau)}\right]\ .
\end{multline*}
At this stage, take
\[
a\geq C\left(2\max\{\|u_T\|_{L^p(\T^d)},\|u_T\|^{\gamma}_{W^{2-2/p,p}(\T^d)}\}+\|V\|_{L^p(Q_\tau)}\right)+2
\]
to get
\[
\|w\|_{W^{2,1}_p(Q_\tau)}\leq C\left\{\max\{\|z\|_{W^{2,1}_p(Q_\tau)},\|z\|^{\gamma}_{W^{2,1}_p(Q_\tau)}\}\left[(1+C(\delta))\tau^{\frac{1}{2p}}+\delta\right]\right\}+a-2\ .
\]
Then, choose $\delta\leq\frac{1}{Ca}$ so that
\[
\|w\|_{W^{2,1}_p(Q_\tau)}\leq C\max\{\|z\|_{W^{2,1}_p(Q_\tau)},\|z\|^{\gamma}_{W^{2,1}_p(Q_\tau)}\}(1+C(\delta))\tau^{\frac{1}{2p}}+a-1
\]
and finally $\tau$ small to conclude
\[
\|w\|_{W^{2,1}_p(Q_\tau)}\leq a\ .
\]
This shows that $J$ maps $\mathcal{S}_a$ into itself.

To prove that $J$ is a contraction, one has to argue as above, exploiting also the fact that
for bounded $z\in W^{2,1}_p(Q_T)$, $p>d+2$, then $Dz$ is bounded in $L^{\infty}(Q_T)$. So,
\begin{equation*}
\norm{H(x,Dz_1)-H(x,Dz_2)}_{L^{p}(Q_T)}\leq C\norm{z_1-z_2}_{L^p(Q_T)}\leq CT\norm{\partial_t(z_1-z_2)}_{L^p(Q_T)}
\end{equation*}
for some positive constant $C$. Therefore, one obtains, for small $\tau$,
\begin{equation*}
\norm{Jz_1-Jz_2}_{W^{2,1}_p(Q_{\bar{T}})}\leq\frac12\norm{z_1-z_2}_{W^{2,1}_p(Q_{\bar{T}})},
\end{equation*}
which ensures the existence of a unique fixed point, $z = Jz$, i.e. a solution $z$ of the HJB equation in the interval $(T-\tau, T]$.\\
Now note that by Sobolev embedding, if $p>d+2$, then $u\in C^{1+\alpha,\frac{1+\alpha}{2}}(Q_\tau)$. Then a bootstrap argument allows to conclude $u\in C^{4+\alpha,2+\alpha/2}(Q_\tau)$, since $V\in C^{2+\alpha,1+\alpha/2}(Q_\tau)$.\\
\par\smallskip
\textit{Step 2}. Define
\begin{equation*}
T^*:=\inf\{\tau\in[0,T]:\eqref{HJ} \text{ admits a solution $C^{4+\alpha,2+\alpha/2}(Q_{\tau})$}\}
\end{equation*}
In view of Step 1 we claim that the above set is nonempty. We want to show that $T^* \le 0$. To this aim, take a sequence $\{(\tau_k,u_k)\}$ in $(T^*,T)\times W^{2,1}_p(\overline{Q}_{\tau_k})$, where $\tau_k$ converges decreasingly to $T^*$ and $u_k$ solves \eqref{HJ} in $\overline{Q}_{\tau_k}$. Since, by Sobolev Embedding, $u_k \in C^{4+\alpha,2+\alpha/2}(Q_{\tau_k})$, we have that $u_k$ is semiconcave independently on $k$. Being also bounded by the Comparison Principle for classical solutions of integro-differential uniformly parabolic equations (see \cite[Corollary II.2.18]{GMparabolic}, there exists $C>0$ such that
\begin{equation*}
\norm{Du_k}_{L^{\infty}(Q_{\tau})}\leq C \qquad \forall k\in\N
\end{equation*}
(see \cite[Remark 2.1.8]{CS}). Arguing as in Step 1, by \cite[Theorem IV.9.1]{LSU} we claim that $u_k$ satisfies
\begin{equation}\label{1}
\norm{u_k}_{W^{2,1}_p(Q_{\tau})}\leq C.
\end{equation}
In particular the solution turns out to be classical by bootstrapping and \cite[Theorem II.3.1]{GMparabolic}. Again by the Comparison Principle, we also have
\begin{equation}
\label{2}
u_k=u_h\text{ on }\overline{Q}_{\tau_h}\text{ for every }k\geq h\ .
\end{equation}
We define a function $u:\T^d\times[T^*,T]\rightarrow\R$ by setting $u=u_k$ on $\overline{Q}_{\tau_k}$ for every $k\in\N$ and then by taking its continuous extension to $\T^d\times[T^*,T]$. Moreover, it solves the Cauchy problem on $\T^d\times[T^*,T]$ by continuity of $u,\partial_tu,Du,D^2u$ (using the results for parabolic H\"older spaces, since, as claimed above, at the end $u$ has classical regularity). If, by contradiction, $T^*>0$, one argues as in Step 1 to find $w\in W^{2,1}_p(Q_{\tau})$ which solves
\begin{equation*}
-\partial_t w-\Delta w+(-\Delta)^sw+H(x,Dw)=V\text{ on }Q_{\tau}, \qquad w(\cdot,T)=u(\cdot,T^*)\text{ on }\T^d
\end{equation*}
(basically one applies the local existence to the backward equation with datum in $T^*$) which at the end will have $C^{4+\alpha,2+\alpha/2}$ regularity. One can check that
\begin{equation*}
u^*(x,t)=\begin{cases}
u(x,t)&\text{ if }(x,t)\in \T^d\times[T^*,T]\ ,\\
w(x,T+t-T^*)&\text{ if }(x,t)\in\T^d\times[T^*-\tau,T^*]
\end{cases}
\end{equation*}
belongs to $C^{4+\alpha,2+\alpha/2}(\T^d\times[T^*-\tau,T])$ and solves the problem on $\T^d\times[T^*-\tau,T]$, contradicting the minimality of $T^*$.\\
\end{proof}

\section{Existence for the MFG system}\label{SecMFG}

This section is devoted to the proofs of existence for systems \eqref{fmfg} and \eqref{fmfgv}. We begin by the viscous case, then proceed with the vanishing viscosity procedure.

\subsection{The viscous case}
\begin{proof}[Proof of Theorem \ref{sigmapositivo}]
The statement is a consequence of the Schauder's fixed point theorem (see \cite[Corollary 11.2]{GT}). Let 
\[
\mathcal{X}=C^{1+\alpha/2}([0,T];\mathcal{P}(\T^d))
\]
and
\[
\mathcal{C}=\{m\in\mathcal{X}:\|m\|_{C^{1+\alpha/2}([0,T];\mathcal{P}(\T^d))}\leq \overline C\}.
\]
It is straightforward to see that $\mathcal{C}$ is closed and convex. We construct a map $T:\mathcal{C}\rightarrow \mathcal{C}$ in the following way: given $\mu\in \mathcal{C}$, let $u$ be the unique solution to
\begin{equation}\label{HJmfg}
\begin{cases}
-\partial_tu-\sigma\Delta u+(-\Delta)^s u+D_pH(x,Du)=F[\mu(t)](x) & \text{ in }\T^d\times(0,T)\ ,\\
u(x,T)=u_T(x) & \text{ in }\T^d\ .
\end{cases}
\end{equation}
Then we define $m=T(\mu)$ as the solution to the fractional Fokker-Planck equation
\begin{equation}\label{FPmfg}
\begin{cases}
\partial_tm-\sigma\Delta m+(-\Delta)^s m-\dive(mD_pH(x,Du))=0 & \text{ in }\T^d\times(0,T)\ ,\\
m(x,0)=m_0(x) & \text{ in }\T^d\ .
\end{cases}
\end{equation}
We divide the proof in three steps.\\
\par\smallskip
\textit{Step 1. $T$ is well-defined}. To show that the map $T$ is well-defined, first note that, since $\mu\in C^{1+\alpha/2}_t(Q_T)$, by the assumptions on $F$ we have $F[\mu]\in C^{2+\alpha,1+\alpha/2}(Q_T)$; in particular, $F[\mu]$ is bounded in $C^{2+\alpha,1+\alpha/2}(Q_T)$ independently with respect of $\mu$. By Proposition \ref{existenceHJ}, problem \eqref{HJmfg} has a unique classical solution belonging to $C^{4+\alpha,2+\alpha/2}(Q_T)$, and satisfies the a priori estimate
\begin{equation*}
\norm{u}_{C^{4+\alpha,2+\alpha/2}(Q_T)}\leq C_1
\end{equation*}
where $C_1$ in particular depends on $\norm{u_T}_{C^{4+\alpha}(\T^d)}$, but does not depend on $\mu$. Then, we can expand the divergence term of the viscous fractional Fokker-Planck equation as
\begin{equation*}
\partial_tm-\sigma\Delta m+(-\Delta)^s m-D_pH(x,Du)\cdot Dm-m\dive(D_pH(x,Du))=0\ ,
\end{equation*}
which turns out to be a linear equation with parabolic H\"older coefficients in $C^{2+\alpha,1+\alpha/2}(Q_T)$, uniformly with respect to $\mu$. Indeed $\dive(D_pH(x,Du))\in C^{2+\alpha,1+\alpha/2}(Q_T)$ owing to \cite[Remark 8.8.7]{KrylovbookHolder}. This gives that 
\begin{equation}\label{mbou}
\|m\|_{C^{4+\alpha,2+\alpha/2}(Q_T)} \le C_2 
\end{equation} 
by \cite[Theorem II.3.1]{GMparabolic}. In particular, the map $T$ is well-defined from $\mathcal{C}$ into itself by choosing $\overline C$ above large enough.\\
\par\smallskip
\textit{Step 2. T is continuous}. To this aim, let $\mu_n\in \mathcal{C}$ converging to some $\mu$. Let $(u_n,m_n),(u,m)$ be the corresponding solutions. By the continuity assumption (F1) we conclude that the map $(x,t)\longmapsto F[\mu_n(t)](x)$ uniformly converges to $(x,t)\longmapsto F[\mu(t)](x)$. We can then consider the equation
\begin{equation*}
-\partial_tu_n-\sigma\Delta u_n+(-\Delta)^su_n+H(x,Du_n)=F[\mu_n(t)](x)
\end{equation*}
whose right-hand side $F[\mu_n(t)](x)$ is uniformly bounded in $C^{2+\alpha,1+\alpha/2}(Q_T)$. Then the sequence  $\{u_n\}$ is uniformly bounded in $C^{4+\alpha,2+\alpha/2}(Q_T)$ in view of Proposition \ref{existenceHJ} and thus converges in $C^{4,2}$ to the unique solution $u$ of the HJB equation. As before, the $m_n$ are solutions of a linear equation with H\"older continuous coefficients, providing uniform estimates in $C^{4+\alpha,2+\alpha/2}(Q_T)$ for $\{m_n\}$. Therefore $\{m_n\}$ converges in $C^{4,2}$ to the unique solution $m$ of the Fokker-Planck equation. Note that the convergence holds also in $\mathcal{C}$.
\par\smallskip
\textit{Step 3. $\overline{T(\mathcal{C})}$ is compact.} By bounds \eqref{mbou}, one proves that for every $\mu_n\in \mathcal{C}$, the sequence $m_n = T(\mu_n)$ has a convergent subsequence.
\end{proof}

\subsection{The vanishing viscosity limit}\label{SecVV}

We emphasize that in the limiting procedure $\sigma \to 0$, one passes from classical parabolic $W^{2,1}_p$ regularity to fractional parabolic $\H_p^{2s}(Q_T)$ regularity. The strategy will thus be to pass to the limit in some suitable weak sense, and then recover maximal regularity by means of Theorem \ref{fracregtor}.


\begin{proof}[Proof of Theorem \ref{vanishingviscosity}]
Let $(u_{\sigma},m_{\sigma})$ a the solution of \eqref{fmfgv}. For $\sigma>0$ we know that a solution exists in view of Theorem \ref{sigmapositivo}.
Collecting the results in Proposition \ref{stabFP}, Proposition \ref{semiconcav} and Proposition \ref{hc}, we are able to construct a sequence $\sigma=\{\sigma_n\}\rightarrow0$ such that, if $(u_{\sigma},m_{\sigma})$ is the corresponding solution, we have
\begin{itemize}
\item[(i)] $u_{\sigma}$ converges to $u$ in $C(Q_T)$ as a consequence of the estimate \eqref{holder} and Ascoli-Arzel\'a Theorem. Moreover, one easily has bounds for $u_\sigma$ in $\mathcal{H}^s_2$, so $u_\sigma \to u$ weakly in $\mathcal{H}^s_2$.
\item[(ii)]  The semiconcavity estimates in Proposition \ref{semiconcav} yield $Du_{\sigma}\rightarrow Du$ a.e. in $Q_T$ in view of \cite[Theorem 3.3.3]{CS}. In addition, by \cite[Remark 2.1.8]{CS} they also imply uniform bounds for $Du_{\sigma}$ in $L^{\infty}(Q_T)$, so $Du_{\sigma} \to Du$ in the $L^\infty$-weak-$\ast$ sense.  Finally, $u$ is semiconcave with the same semiconcavity bounds.
\item[(iii)] By (ii) and dominated convergence theorem $Du_{\sigma}\to Du$ in $L^p(Q_T)$ for every finite $p\geq1$.
\item[(iv)] As a consequence of the semiconcavity estimates we have $[\dive(b)]^-\leq C$, where $b=-D_pH(x,Du_{\sigma})$. Indeed
\[
\dive(-D_pH(x,Du_{\sigma}))=-\sum_{i,j}D^2_{p_ix_j}H-\sum_{i,j}D^2_{p_ip_j}H\partial_{x_ix_j}u_{\sigma} \ge -\overline C.
\]
The first term can be controlled by (ii) and (H4). Since $0 \le D^2_pH(x,Du) \le C_1\, I_d$ and $D^2u_{\sigma}\leq CI_d$, we have a control on the second term by a constant independent of $\sigma$.

\item[(v)] In view of the estimate \eqref{minfty}, $m_{\sigma}$ converges to $m\in L^{\infty}(Q_T)$, weakly-$\ast$ in $L^{\infty}$.
\item[(vi)] Proposition \ref{stabFP} ensures that $m_\sigma$, $\partial_tm_\sigma$ are bounded uniformly with respect to $\sigma$ in $\mathbb{H}_2^{s}(Q_T)$ and $\mathbb{H}_2^{-1}(Q_T)$ respectively, so they weakly converge. 
\end{itemize}
In addition, note that $(x,t)\longmapsto F[m_{\sigma}(t)](x)$ uniformly converges to the map $(x,t)\longmapsto F[m(t)](x)$.
We now pass to the limit in the weak formulation of both equations.\\
\par\smallskip
\textit{Step 1. Fokker-Planck Equation}. Multiplying the Fokker-Planck equation by a test function $\varphi\in C^{\infty}(\T^d\times[0,T))$ and integrating over $Q_T$ we get
\begin{multline}\label{intpartslimit}
-\int_{\T^d} m_{\sigma}(x,0)\varphi(x,0)dx-\iint_{Q_T} m_{\sigma}\partial_t\varphi dxdt-\sigma\iint_{Q_T} m_{\sigma}\Delta\varphi dxdt \\+\iint_{Q_T} (-\Delta)^{s/2}m_{\sigma}(-\Delta)^{s/2}\varphi dxdt
+\iint_{Q_T} m_{\sigma} D_pH(x,Du_{\sigma})\cdot D\varphi dxdt=0
\end{multline}

We then let $\sigma\rightarrow0$ to conclude
\begin{multline*}
-\int_{\T^d} m(x,0)\varphi(x,0)dx-\iint_{Q_T} m\partial_t\varphi dxdt+\iint_{Q_T} (-\Delta)^{s/2}m(-\Delta)^{s/2}\varphi dxdt+\\
+\lim_{\sigma \to 0}\iint_{Q_T} m_\sigma D_pH(x,Du_\sigma)\cdot D\varphi dxdt=0\ ,
\end{multline*}
by the convergence of $m_{\sigma}$ stated in (v)-(vi). It remains to prove
\begin{equation*}
\iint_{Q_T} m_{\sigma} D_pH(x,Du_{\sigma})\cdot D\varphi dxdt\rightarrow \iint_{Q_T} m D_pH(x,Du)\cdot D\varphi dxdt\ .
\end{equation*}
We write
\begin{multline*}
\left|\iint_{Q_T} (m_{\sigma} D_pH(x,Du_{\sigma})-m D_pH(x,Du))\cdot D\varphi\, dxdt\right|\leq\\
\leq \iint_{Q_T}\left|m_{\sigma} D_pH(x,Du_{\sigma})-m_{\sigma} D_pH(x,Du)\right| |D\varphi|\, dxdt\\
+\iint_{Q_T}\left|m_{\sigma} D_pH(x,Du)-m D_pH(x,Du)\right||D\varphi|\,dxdt\ .
\end{multline*}
The first term on the right-hand side of the above inequality can be handled using (H2) and (iii)-(v)
\begin{multline*}
\iint_{Q_T}\left|m_{\sigma}(D_pH(x,Du_{\sigma})-D_pH(x,Du))\right||D\varphi|dxdt \\
\leq C \norm{m_{\sigma}}_{L^{\infty}(Q_T)}\norm{D_pH(x,Du_{\sigma})-D_pH(x,Du)}_{L^1(Q_T)} \\
\leq C_1\norm{|Du_{\sigma}|^{\gamma-1}+|Du|^{\gamma-1}}_{L^{p}(Q_T)}\norm{Du_{\sigma}-Du}_{L^q(Q_T)}\leq C_2 \norm{Du_{\sigma}-Du}_{L^q(Q_T)}\ ,
\end{multline*}
where we also applied H\"older's inequality with $p$ conjugate exponent of $q$. Finally,
\begin{equation*}
\iint_{Q_T}(m_{\sigma}-m)D_pH(x,Du)\cdot D\varphi dxdt\rightarrow0
\end{equation*}
in view of the $L^{\infty}$ weak-$\ast$ convergence $m_{\sigma}$ to $m$ and the fact that 
\[
\norm{D_pH(x,Du)}_{L^1(Q_T)}\leq C\norm{Du}^{\gamma-1}_{L^{\gamma-1}(Q_T)}<\infty\ .
\]
\par\smallskip
\textit{Step 2. The HJB equation.} We now pass to the limit in the fractional HJB equation. Multiplying the equation satisfied by $u_{\sigma}$ by a test function $\varphi\in C^{\infty}(\T^d\times(0,T])$ we get
\begin{multline*}
-\iint_{Q_T}\partial_tu_{\sigma}\varphi dxdt-\sigma\iint_{Q_T}\Delta u_{\sigma}\varphi dxdt+\iint_{Q_T}(-\Delta)^{s}u_{\sigma}\varphi dxdt\\
+\iint_{Q_T}H(x,Du_{\sigma})\varphi dxdt=\iint_{Q_T}F[m_{\sigma}(t)]\varphi dxdt
\end{multline*}
We now integrate by parts using Lemma \ref{intparts} to obtain
\begin{multline*}
-\int_{\T^d}u_{\sigma}(x,T)\varphi(x,T)dx+\iint_{Q_T}u_{\sigma}\partial_t\varphi dxdt+\sigma\iint_{Q_T}D u_{\sigma}\cdot D\varphi dxdt\\+\iint_{Q_T}(-\Delta)^{\frac{s}{2}}u_{\sigma}(-\Delta)^{\frac{s}{2}}\varphi dxdt
+\iint_{Q_T}H(x,Du_{\sigma})\varphi dxdt=\iint_{Q_T}F[m_{\sigma}(t)]\varphi dxdt.
\end{multline*}
Now note that (iii) together with Lemma \ref{LpestimatesT} implies also that $(-\Delta)^{\frac{s}{2}}u_{\sigma}\to (-\Delta)^{\frac{s}{2}}u$ in $L^p(Q_T)$. By the regularity assumptions of the coupling $F$, the term on the right-hand side converges to $\iint_{Q_T}F[m(t)]\varphi dxdt$ as $\sigma\rightarrow0$. We only need to prove that
\begin{equation*}
\iint_{Q_T}H(x,Du_{\sigma})\varphi dxdt\rightarrow \iint_{Q_T}H(x,Du)\varphi dxdt
\end{equation*}
as $\sigma\rightarrow0$. To this aim we use again (H2) 
and the convergence of $Du_{\sigma}$ to $Du$ in $L^p$ for every finite $p\geq1$.

\par\smallskip
\textit{Step 3}. Recall that the energy solution $u\in\H_2^s(Q_T)$ of the fractional Hamilton-Jacobi equation is unique. The same is true for the solution of the Fokker-Planck equation in view of Remark \ref{uniqueFP}. Moreover, since $u, m$ are in $L^\infty(Q_T)$, the HJB and Fokker-Planck equations can be considered as fractional heat equations with bounded source terms. Therefore, by Theorem \ref{fracregtor} the solution $u$ belongs a posteriori to $\H_p^{2s}(Q_T)$, while $m$ is of class $\H_p^{2s-1}(Q_T)$ for all $p > 1$. 
\par\smallskip
\textit{Step 4}. Finally, if $s > 1/2$ one can set up a bootstrap procedure to obtain classical regularity. This will be proven in the following Theorem \ref{addreg}.

\end{proof}

\begin{rem}
By uniform convergence of $u_\sigma$ and $F[m_\sigma]$ on $Q_T$ we can also conclude that the limit $u$ solves the HJB equation in \eqref{fmfg} in the viscosity sense.
\end{rem}

\subsection{Classical regularity in the subcritical case $s > 1/2$}

In what follows, we will assume that
\[
\frac12 < s < 1.
\]
We aim at proving that $(u, m)$ previously found in Theorem \ref{vanishingviscosity} solves the MFG system in the classical sense. We stress that for a (linear) bootstrap procedure to be performed, $s$ must be greater than $1/2$, because the Hamiltonian and divergence terms deteriorate the regularity of the unknowns up to one derivative, while the gain realized by the fractional Laplacian is of order $2s$.

\begin{thm}\label{addreg}
Let $s\in(\frac12,1)$ and  $(u, m)$ be a solution to  \eqref{fmfgv} (in the sense of Definitions \ref{defFP} and \ref{weak}). Then $u, m$ both satisfy \eqref{holderb} for some $0 < \bar\alpha < 1$, and  in particular solve \eqref{fmfgv} in the classical sense. Moreover, there exists a constant $C > 0$ depending on the data and remaining bounded for bounded values of $T$ such that
\[
\|m\|_{\infty} + \|Du\|_{\infty} \le C.
\]
\end{thm}

\begin{proof}[Proof of Theorem \ref{addreg}]
We first observe that since $m \in \mathcal{H}_p^{2s-1}(Q_T)$ for all $p > 1$, by Proposition \ref{embe3} we have that $m$ is bounded in $\mathcal{C}^{\bar\alpha,\frac{\bar\alpha}{2s}}(Q_T)$ for some $0 < \bar\alpha < 1$, by choosing $p$ large enough. Therefore, in view of (F3), $F[m] \in C^{\bar\alpha/2s}([0,T]; C^{2+\alpha}(\T^d))$, that is in turn embedded in $\mathbb{H}^2_p$ for all $p > 1$.

Note that $u$ solves the following equation
\begin{equation*}
-\partial_tu+(-\Delta)^su=G(x,t), \qquad u(x,T) = u_T(x),
\end{equation*}
where $G(x,t):=F[m(t)](x)-H(x,Du(x,t))$, and $Du \in L^{\infty}$. Then, at first glance, $G \in L^p(Q_T)$ for all $p$. This yields $u \in \mathcal{H}^{2s}_p(Q_T)$ by applying Theorem \ref{fracregtor}, and in particular $Du \in \mathbb{H}^{2s-1}_p(Q_T)$. Then $H(x,Du)\in \mathbb{H}^{2s-1-\eps}_p(Q_T)$ by the fractional chain rule in Lemma \ref{chain}, so $G \in \mathbb{H}^{2s-1-\eps}_p(Q_T)$.
 Using that $s>\frac12$ and taking $\eps$ small, we can iterate this procedure until, in a finite number of steps, $G \in \mathbb{H}^{2}_p(Q_T)$, that is the maximal regularity allowed by $F[m] \in \mathbb{H}^2_p(Q_T)$. Another iteration yields  $u \in \mathcal{H}^{2+2s}_p(Q_T)$ for all $p > 1$. Since $2+2s> 3$, we can apply Theorem \ref{Embedding} with $p$ large and $\beta$ close to zero to obtain $u \in C^{\alpha_1}([0,T]; C^{3+\alpha_2}(\T^d))$, for some $0 < \alpha_1, \alpha_2 < 1$, thus $H(x, Du) \in C^{\alpha_1}([0,T]; C^{2+\alpha_2}(\T^d))$. As a consequence, $G \in \mathcal{C}^{\bar\alpha,\frac{\bar\alpha}{2s}}(Q_T)$, possibly for a smaller $\bar\alpha$ than the one appeared at the beginning of the proof. So, Theorem \ref{Regularity} applies, providing the desired regularity for $u$.

Let us now focus on the Fokker-Planck equation. By similar arguments we have that $D_p H(x, Du)$ $\in \mathbb{H}^{1+2s-\eps}_p \cap L^\infty(Q_T)$. Moreover, $m \in \mathcal{H}_p^{2s-1} \cap L^\infty(Q_T)$, so by Lemma \ref{katoponce} we obtain that ${\rm div} (m D_p H(x, Du)) \in \mathbb{H}_p^{2s-2}(Q_T)$. An application of fractional parabolic regularity stated in Theorem \ref{fracregtor} provides $m \in  \mathcal{H}_p^{4s-2}(Q_T)$. We may iterate this procedure until we get $m \in \mathcal{H}_p^{2s+1} \cap L^\infty(Q_T)$, and another time to conclude  {$m \in \mathcal{H}_p^{4s-\eps}(Q_T)$} for all $p > 1$. Since $4s > 2$, we can use Theorem \ref{Embedding} with $p$ large and $\beta$ small to get $m \in C^{\alpha_3}([0,T]; C^{1+\alpha_4}(\T^d))$, for some $0 < \alpha_3, \alpha_4 < 1$. Since we previously obtained $D_p H(x, Du) \in C^{\alpha_1}([0,T]; C^{2+\alpha_2}(\T^d))$, we finally have ${\rm div} (m D_p H(x, Du))  \in \mathcal{C}^{\bar\alpha,\frac{\bar\alpha}{2s}}(Q_T)$, reducing eventually the value of $\bar\alpha$ previously chosen. We deduce the stated regularity for $m$ again from Theorem \ref{Regularity}.

Last, the estimate on the sup-norm of $D u$ on $Q_T$ follows by comparison and semiconcavity bounds. Note that Proposition \ref{semiconcav} applies in view of $C^{\alpha_1}([0,T]; C^{3+\alpha_2}(\T^d))$ regularity of $u$, see in particular Remark \ref{fracdiff}. Analogous bounds  for $m$ are then a direct consequence of Theorems \ref{fracregtor} and \ref{Embedding}.

\end{proof}

\begin{rem} We mention that if $u_T, m_0, H$ and $F$ are smoother, an additional bootstrap procedure yields further regularity of $u, m$, up to $C^\infty$. For the sake of brevity, we omit the details.
\end{rem}

\section{Uniqueness}\label{SecUn}

Here, we prove some uniqueness results in the case $\sigma = 0$, that is for system \eqref{fmfg}. We assume that equations are satisfied in the sense of Definitions \ref{defFP} and \ref{weak}. The case $\sigma > 0$ is easier, since solutions enjoy classical regularity, and the following arguments apply similarly.

\subsection{Uniqueness in the monotone case}
\begin{thm}\label{un1}
Assume that $H$ is convex and the following monotonicity condition holds
\begin{equation*}
\label{condunique}
\int_{\T^d}(F[m_1](x)-F[m_2](x))d(m_1-m_2)(x)>0\ , \quad \forall m_1,m_2\in\mathcal{P}(\T^d)\ ,m_1\neq m_2\ .
\end{equation*}
Then, the solution to \eqref{fmfg} is unique.
\end{thm}

\begin{proof}
Uniqueness in the monotone case follows from the usual ideas by Lasry-Lions \cite{ll}. One has to be careful that $(u, m)$ is regular enough to run the argument. Let $(u_1,m_1)$ and $(u_2,m_2)$ be two solutions of the MFG system \eqref{fmfg}. Set $v=u_1-u_2$ and $\mu=m_1-m_2$. Then $v$ and $\mu$ satisfy respectively the equations
\begin{equation*}
-\partial_tv+(-\Delta)^s v+H(x,Du_1)-H(x,Du_2)=F[m_1(t)](x)-F[m_2(t)](x)\ ,v(x,T)=0
\end{equation*}
and
\begin{equation*}
\partial_t\mu+(-\Delta)^s\mu-\dive\big(m_1D_pH(x,Du_1)-m_2D_pH(x,Du_2)\big)=0\ ,\mu(x,0)=0\ .
\end{equation*}
We distinguish between the supercritical-critical (namely $s\in(0,1/2)$ and $s=1/2$) case and the subcritical ($s\in(1/2,1)$) one.

\textit{Case 1. The supercritical-critical case}. Recall that $u_i, Du_i, m_i \in L^{\infty}(Q_T)$, so $v, Dv, \mu\in L^{\infty}(Q_T)$. Moreover, $v\in \H_2^s(Q_T)$. Hence, using $\mu\in \mathbb{H}_2^s(Q_T)\cap L^{\infty}(Q_T)$ as a test function in the weak formulation of Definition \ref{weak}, we get
\begin{multline}\label{eq3}
\iint_{Q_T}-\mu\partial_t v+\mu\big(H(x,Du_1)-H(x,Du_2)\big)-\mu\big(F[m_1(t)](x)-F[m_2(t)](x)\big)dxdt+\\
+\iint_{Q_T}(-\Delta)^{\frac{s}{2}}\mu(-\Delta)^{\frac{s}{2}}v\,dxdt=0\ .
\end{multline}
Then, we use $v\in\H_2^s(Q_T) \cap L^\infty(0,T;W^{1,\infty}(\T^d))$ as a test function in the weak formulation of the equation satisfied by $\mu$, recalling also that $\partial_t\mu\in\mathbb{H}_2^{-1}(Q_T)$, to conclude
\begin{multline}\label{eq4}
0=\iint_{Q_T}-\mu \partial_t v\,dxdt+(-\Delta)^{\frac{s}{2}}\mu(-\Delta)^{\frac{s}{2}}v\, dxdt+\\
+\iint_{Q_T} Dv\cdot (m_1D_pH(x,Du_1)-m_2D_pH(x,Du_2))\,dxdt\ ,
\end{multline}
Subtracting \eqref{eq4} from \eqref{eq3} we obtain
\begin{multline}\label{claimuniqueness}
0=\iint_{Q_T}-\mu \big((F[m_1(t)](x)-F[m_2(t)](x)\big)+\mu(H(x,Du_1)-H(x,Du_2))dxdt-\\
-\iint_{Q_T}Dv\cdot \big(m_1D_pH(x,Du_1)-m_2D_pH(x,Du_2)\big)dxdt\ .
\end{multline}
The following inequality holds true
\begin{equation*}
\iint_{Q_T}\mu(H(x,Du_1)-H(x,Du_2))-Dv\cdot (m_1D_pH(x,Du_1)-m_2D_pH(x,Du_2))dxdt\leq0\ ,
\end{equation*}
by convexity of $H$.
Using \eqref{claimuniqueness} we can conclude that
\begin{equation*}
\iint_{Q_T}(m_1 - m_2) \big(F[m_1(t)]-F[m_2(t)]\big)dtdx\leq0\ ,
\end{equation*}
In view of the monotonicity condition we get $m_1=m_2$ a.e.. Finally, by the fact that $u_1$ and $u_2$ solves the same equation with same final datum, they must concide.

\textit{Case 2}.\textit{ The subcritical case}. The proof of the case $s\in(\frac12,1)$ is simpler and it can be carried out as in Step 1, observing that $(u,m)$ is a classical solution. 
\end{proof}
\subsection{Small-time uniqueness}
The result of this section is the following
\begin{thm}\label{smallT}
For $s\in(\frac12,1)$, there exists $T^*>0$, depending on $d,s,H,F,m_0, u_T$ 
such that for all $T\in(0,T^*]$ system \eqref{fmfg} has at most one solution $(u,m)$.
\end{thm}
Rewriting \eqref{fmfg} as a forward-forward system for $v,m$ setting $v(\cdot,t):=u(\cdot,T-t)$ for all $t\in[0,T]$, then
\begin{equation}\label{ffs}
\begin{cases}
v(x,t)=\mathcal{T}_tu_T(x)-\int_0^t\mathcal{T}_{t-\tau}\Phi^v[v,m](\tau)(x)d\tau\ ,\\
m(x,t)=\mathcal{T}_tm_0(x)+\int_0^t\mathcal{T}_{t-\tau}\Phi^m[v,m](\tau)(x)d\tau\ ,
\end{cases}
\end{equation}
where 
\begin{equation*}
\Phi^v[v,m](\tau)(\cdot)=F[m(T-\tau)](\cdot)-H(\cdot,Dv(\cdot,\tau))\ ,
\end{equation*}
\begin{equation*}
\Phi^m[v,m](\tau)(\cdot)=\dive(D_pH(\cdot,Dv(\cdot,T-\tau))m(\tau))
\end{equation*}
for $\tau\in[0,T]$. We will exploit the decay properties of $\mathcal{T}_t$.

\begin{proof}[Proof of Theorem \ref{smallT}]
For $p > 1$ and $\mu \ge 0$, let us denote by
\begin{equation*}
X_p^\mu:=C([0,T];H_p^{\mu}(\T^d)).
\end{equation*}

First, observe that any solution is classical by Theorem \ref{addreg}, and therefore it belongs to $X_p^{2s}\times X_p^{2s-1}$. Moreover, every solution of \eqref{fmfg} can be seen as a fixed point of the map $\Psi:(v,m)\longmapsto (\hat{v},\hat{m})$, where
\begin{equation}\label{mapfixed}
\begin{cases}
\hat{v}(t)=\mathcal{T}_tu_T(x)-\int_0^t\mathcal{T}_{t-\tau}\Phi^v[v,m](\tau)(x)d\tau\ ,\\
\hat{m}(t)=\mathcal{T}_tm_0(x)+\int_0^t\mathcal{T}_{t-\tau}\Phi^m[v,m](\tau)(x)d\tau\ .
\end{cases}
\end{equation}
We prove that the fixed point of $\Psi$ defined in \eqref{mapfixed} is unique by the contraction properties of $\Psi$ itself that are valid for small $T$. Let $(v_1,m_1)$ and $(v_2,m_2)$ be two fixed points of $\Psi$. Set $\epsilon=d\left(\frac1p-\frac{1}{\bar{p}}\right)<2s-1$ with $\bar{p}>p$. This choice yields
\[
\norm{m(\tau)}_{2s-1-\epsilon,\bar{p}}\leq C\norm{m(\tau)}_{2s-1,p}
\]
for some $C>0$ in view of Lemma \ref{inclstat}.
We apply Lemma \ref{decaytorus}-(i) (with $\nu=2s-1-\eps$ and $\gamma=1+\eps$) and the assumptions on $F$ and $H$ to get
\begin{multline*}
\norm{\int_0^t\mathcal{T}_{t-\tau}(\Phi^v[v_1,m_1](\tau)(x)-\Phi^v[v_2,m_2](\tau)(x))d\tau}_{2s,p}\leq\\
\leq  \int_0^t\norm{\mathcal{T}_{t-\tau}(\Phi^v[v_1,m_1](\tau)(x)-\Phi^v[v_2,m_2](\tau)(x))}_{2s,p}d\tau\\
\leq C_1\left(\int_0^t(t-\tau)^{-\frac{1+\eps}{2s}}\norm{F[m_1(T-\tau)](\cdot)-F[m_2(T-\tau)](\cdot)}_{2s-1-\eps,p}d\tau+\right.\\
\left.+\int_0^t (t-\tau)^{-\frac{1+\eps}{2s}} \norm{H(\cdot,Dv_1(\cdot,T-\tau))-H(\cdot,Dv_2(\cdot,T-\tau))}_{2s-1-\eps,p}d\tau\right) \\
\leq C_2\left(\int_0^t(t-\tau)^{-\frac{1+\eps}{2s}}\norm{m_1(\cdot,T-\tau)-m_2(\cdot,T-\tau)}_{2s-1,p}d\tau+\right.\\
\left.+\int_0^t(t-\tau)^{-\frac{1+\eps}{2s}}\norm{Dv_1(\cdot,T-\tau)-Dv_2(\cdot,T-\tau)}_{2s-1,p}d\tau\right) \\
\leq C_3T^{\frac{2s-1-\eps}{2s}}\left(\norm{m_1-m_2}_{X_p^{2s-1}}+\norm{v_1-v_2}_{X_p^{2s}}\right)\ ,
\end{multline*}
by taking $T$ small enough. \\
We now consider the term related to the Fokker-Planck equation. We apply Lemma \ref{decaytorus}-(i) with $\nu=2s-2-\eps$ and $\gamma=1+\eps$ to obtain
\begin{multline*}
\norm{\int_0^t\mathcal{T}_{t-\tau}(\Phi^m[v_1,m_1](\tau)(x)-\Phi^m[v_2,m_2](\tau)(x))d\tau}_{2s-1,p}\le\\
 \le\int_0^t\norm{\mathcal{T}_{t-\tau}(\Phi^m[v_1,m_1](\tau)(x)-\Phi^m[v_2,m_2](\tau)(x))}_{2s-1,p}d\tau \\
\int_0^t(t-\tau)^{-\frac{1+\eps}{2s}}\norm{\dive\big(D_pH(\cdot,Dv_1(\cdot,T-\tau))m_1(\tau)- D_pH(\cdot,Dv_2(\cdot,T-\tau))m_2(\tau)\big)}_{2s-2-\epsilon,p}\\
\leq C_1\left(\int_0^t(t-\tau)^{-\frac{1+\eps}{2s}}\norm{\dive(D_pH(\cdot,Dv_1(\cdot,T-\tau))(m_1(\tau)-m_2(\tau)))}_{2s-2-\epsilon,p}d\tau+\right.\\
\left.+\int_0^t(t-\tau)^{-\frac{1+\eps}{2s}}\norm{\dive(m_2(\tau)(D_pH(\cdot,Dv_1(\cdot,T-\tau))-D_pH(\cdot,Dv_2(\cdot,T-\tau))))}_{2s-2-\epsilon,p}d\tau\right) \\
\leq C_2\left(\int_0^t(t-\tau)^{-\frac{1+\eps}{2s}}\norm{D_pH(\cdot,Dv_1(\cdot,T-\tau))(m_1(\tau)-m_2(\tau))}_{2s-1-\epsilon,p}d\tau+\right.\\
\left.+\int_0^t(t-\tau)^{-\frac{1+\eps}{2s}}\norm{m_2(\tau)(D_pH(\cdot,Dv_1(\cdot,T-\tau))-D_pH(\cdot,Dv_2(\cdot,T-\tau)))}_{2s-1-\epsilon,p}d\tau)\right)
\end{multline*}
Then one has to observe that 
\begin{multline*}
\norm{D_pH(Dv_1(\cdot,T-\tau))(m_1(\tau)-m_2(\tau))}_{2s-1-\epsilon,p} \\
\leq C_3(\norm{D_pH}_{\bar{q}}\norm{m_1-m_2}_{2s-1-\epsilon,\bar{p}}+\norm{D_pH}_{2s-1-\epsilon,\bar{q}}\norm{m_1-m_2}_{\bar{p}}) \\
\leq C_4\norm{m_1-m_2}_{2s-1-\epsilon,\bar{p}}\leq C_5\norm{m_1-m_2}_{2s-1,p}\ ,
\end{multline*}
where we applied Lemma \ref{katoponce} to the second inequality, Lemma \ref{inclstat}-(iii) to the last one, the fact that $\norm{D_pH}_{2s-1-\epsilon,\bar{q}}$ is bounded independently of $T$ by the regularity assumption on $H$ and the $L^{\infty}$ bound on $Du$ and $m$.

Similarly, 
\begin{multline*}
\norm{m_2(\tau)(D_pH(\cdot,Dv_1(\cdot,T-\tau))-D_pH(\cdot,Dv_2(\cdot,T-\tau)))}_{2s-1-\epsilon,p} \\
\leq C_1\Big(\norm{m_2}_{\bar{q}}\norm{D_pH(\cdot,Dv_1)-D_pH(\cdot,Dv_2)}_{2s-1-\epsilon,\bar{p}}+\\ \qquad\qquad\qquad+\norm{m_2}_{2s-1-\epsilon,\bar{q}}\norm{D_pH(\cdot,Dv_1)-D_pH(\cdot,Dv_2)}_{\bar{p}}\Big) \\
\leq C_2\norm{D_pH(\cdot,Dv_1)-D_pH(\cdot,Dv_2)}_{2s-1-\epsilon,\bar{p}}\leq C_3\norm{D(v_1-v_2)}_{2s-1-\epsilon,\bar{p}} \\
\leq C_4\norm{D(v_1-v_2)}_{2s-1,p}\leq C_{5}\norm{v_1-v_2}_{2s,p},
\end{multline*}
where $C_i=C_i(d,s,\epsilon,p,\bar{p},\bar{q})$. This gives
\begin{multline*}
\norm{\int_0^t\mathcal{T}_{t-\tau}(\Phi^m[v_1,m_1](\tau)(x)-\Phi^m[v_2,m_2](\tau)(x))d\tau}_{2s-1,p} \\
\leq C_4T^{\frac{2s-1-\eps}{2s}}(\norm{v_1-v_2}_{X_p^{2s}}+\norm{m_1-m_2}_{X_p^{2s-1}})
\end{multline*}
by eventually taking $T$ small enough. At the end we get
\begin{multline*}
\norm{v_1-v_2}_{X_p^{2s}}+\norm{m_1-m_2}_{X_p^{2s-1}}=\|\Psi(v_1,m_1)-\Psi(v_2,m_2)\|_{X_p^{2s}\times X_p^{2s-1}} \\
\leq \frac{1}{2}(\norm{v_1-v_2}_{X_p^{2s}}+\norm{m_1-m_2}_{X_p^{2s-1}})\ ,
\end{multline*}
which allows to conclude $(v_1,m_1)=(v_2,m_2)$ for $T$ sufficiently small.
\end{proof}

\appendix
\section{Fractional product and chain rules on the torus}

We first present a version of the  Kato-Ponce inequality on Bessel potential spaces on the torus. We refer the reader to the classical results in \cite{KPV} and to \cite{Grafakos14} (and references therein) for more recent developments, all stated in the euclidean case.
\begin{lemma}\label{katoponce}
Let $\mu\in(0,1)$ and $1<p,p_1,q_1,p_2,q_2<\infty$ and such that $\frac1p=\frac{1}{p_1}+\frac{1}{q_1}=\frac{1}{p_2}+\frac{1}{q_2}$. Then,
\begin{equation*}
\norm{fg}_{H^{\mu}_p(\T^d)}\leq C( \norm{f}_{L^{p_1}(\T^d)}\norm{g}_{H^{\mu}_{q_1}(\T^d)}+\norm{f}_{H^{\mu}_{p_2}(\T^d)}\norm{g}_{L^{q_2}(\T^d)})
\end{equation*}
for some $C > 0$.
\end{lemma}
We recall that the inequality can be proven in the euclidean case as follows, see e.g. \cite{GrafakosPDE}. 
First, a bilinear multiplier operator with symbol $m$ acting on $f,g\in\mathcal{S}(\R^d)$ is defined as
\begin{equation}\label{cmR}
T_m(f,g)(x):=\iint_{\R^{2d}}m(\xi,\eta)\mathcal{F}f(\xi)\mathcal{F}g(\eta)e^{2\pi i(\xi+\eta)\cdot x}d\xi d\eta\ .
\end{equation}
We are interested in the symbol $|\xi + \eta|^\mu$, since
\[
(-\Delta)^{\mu/2}(fg)(x)=\iint_{\R^{2d}}|\xi+\eta|^{\mu}\mathcal{F}f(\xi)\mathcal{F}g(\eta)e^{2\pi i(\xi+\eta)\cdot x}d\xi d\eta\ .
\]
Then one performs the partition $m(\xi,\eta) = \sigma_1(\xi,\eta) |\xi|^\mu + \sigma_2(\xi,\eta) |\eta|^\mu$, where
\[
\sigma_1(\xi,\eta):=\frac{|\xi+\eta|^{\mu}}{|\xi|^\mu}\left(1-\phi\left(\frac{|\xi|}{|\eta|}\right)\right), \qquad
\sigma_2(\xi,\eta):=\frac{|\xi+\eta|^{\mu}}{|\eta|^\mu}\phi\left(\frac{|\xi|}{|\eta|}\right)
\]
and $\phi$ is a suitable $C_0^{\infty}$ cut-off function; we are then reduced to prove the boundedness of the operators $T_{\sigma_i}$ on $L^{p_i}(\R^d) \times L^{q_i}(\R^d)$. Indeed, this would yield
\[
\norm{(-\Delta)^{\mu/2}(fg)}_{L^p(\R^d)}\leq C\left( \norm{(-\Delta)^{\mu/2}f}_{L^{p_1}(\R^d)}\norm{g}_{L^{q_1}(\R^d)}+\norm{f}_{L^{p_2}(\R^d)}\norm{(-\Delta)^{\mu/2}g}_{L^{q_2}(\R^d)}\right),
\]
and the desired estimate with $H^\mu_p$ norms would follow by equivalence of $\|\cdot\|_{\mu,p}$ with $\|\cdot\|_{p} + \|(-\Delta)^{\frac{\mu}{2}}\cdot\|_{p}$.
The key result for boundedness of $T_{\sigma_i}$ is the Coifman-Meyer multiplier theorem (see \cite[Theorem A]{Grafakos14} and references therein). Note that the assumptions of such theorem are fulfilled, since the multipliers $\sigma_i$ are homogeneous of degree zero.
\begin{proof}[Proof of Lemma \ref{katoponce}]
One may argue as in the euclidean case. We start by observing that bilinear operators $T_{\sigma_i}$ have a periodic counterpart defined on the torus, that is
\begin{equation}\label{cfT}
B_{\sigma_i}(f,g)(x):=\sum_{\mu\in\Z^d}\sum_{\nu\in\Z^d}\sigma_i(\mu,\nu)\hat f(\mu)\hat g(\nu)e^{2\pi i(\mu+\nu)\cdot x}
\end{equation}
By the transference results on multilinear multipliers in \cite[Theorem 3]{FS}, since $\sigma_i$ are bilinear Coifman-Meyer multipliers on $\R^d\times\R^d$, then they are so also on $\T^d\times\T^d$. One has just to be careful since $\sigma_i$ are discontinuous at $(0, 0)$, but it is sufficient to have them defined in $(0, 0)$ so that $(0, 0)$ is a Lebesgue point for both $\sigma_i$.
\end{proof}
We also present a chain rule for fractional Sobolev spaces. 
\begin{lemma}\label{chain}
Let $\mu > 0$, and $\Psi : \T^d\times\R^d\to\R$ be of class $C^{\lceil{\mu}\rceil}(\T^d\times\R^d)$ with bounded derivatives on $\T^d\times\R^d$ up to order $\lceil{\mu}\rceil$. Let $u\in W^{\mu,p}(\T^d) \cap H^\mu_p(\T^d)$. Then
\[
\|\Psi(\cdot,u(\cdot))\|_{W^{\mu,p}(\T^d)}\leq C(\|u\|_{W^{\mu,p}(\T^d)} + 1)\ ,
\]
and, for all $\eps > 0$,
\[
\|\Psi(\cdot,u(\cdot))\|_{H^{\mu-\eps}_p(\T^d)}\leq C(\|u\|_{H^\mu_p(\T^d)} + 1)\ .
\]
\end{lemma}
\begin{proof} We just consider the case $0 <\mu < 1$, the general case being treated similarly. We start with the inequality in $W^{\mu,p}$ spaces, using their construction through the trace method. It is sufficient to recall that
\[
\|u\|_{W^{1-\mu,p}(\T^d)} = \inf_{u=f(0)} \, \max\{\|t^{\mu-1/p}f(t)\|_{L^p(0,\infty;W^{1,p}(\T^d))};\|t^{\mu-1/p}f'(t)\|_{L^p(\T^d\times(0,\infty))}\},
\]
and observe that
\[
\|\Psi(x, f(x))\|_{W^{1,p}(\T^d)} \le C(1 + \|f\|_{W^{1,p}(\T^d)}),
\]
where the constant $C$ depends on global bounds on the derivatives of $\Psi$. Then, one uses $\Psi(x, f(x))$ to estimate $\|\Psi(\cdot,u(\cdot))\|_{W^{1-\mu,p}(\T^d)}$, where $f$ is close to the infimum in the definition of $\|u\|_{W^{1-\mu,p}(\T^d)}$. The analogous inequality in $H^\mu_p$ spaces is then a consequence of Lemma \ref{inclusioni}.
\end{proof}

\section{Regularity in parabolic fractional H\"older spaces}
We consider the problem
\begin{equation}\label{fracreghol}
\begin{cases}
\partial_tu+(-\Delta)^su=f(x,t)&\text{ in }Q_T\ ,\\
u(x,0)=u_0(x)&\text{ in }\T^d\ .
\end{cases}
\end{equation}
The purpose of this section is to present a fractional analogue of classical parabolic H\"older and Sobolev regularity. We point out that related results for this problem on the euclidean space appeared in \cite[Appendix A]{CF} and \cite{CL}, see also references therein. 
We stress that transference of these results to the periodic setting is delicate, in particular concerning regularity in Sobolev spaces, and to our knowledge they are not explicitly stated in the literature. We present some proofs that make use of interpolation methods and results for abstract parabolic equations, with some details for the reader's convenience. 

As for regularity in H\"older spaces, we follow the approach of \cite[Chapter 5-6]{LunardiSNS},\cite[Chapter 3-4]{LunardiNote} (see also \cite[Chapter 5]{LunardiBook}).
\begin{thm}\label{Regularity}
Let $\alpha\in(0,1)$ so that $2s+\alpha$ is not an integer, $f\in \mathcal{C}^{\alpha,\frac {\alpha} {2s}}(Q_T)$ and $u_0\in C^{2s+\alpha}(\T^d)$. Then problem \eqref{fracreghol} has a unique classical solution $u$, and there exists a positive constant $C$ depending on $d, T, \alpha, s$ 
(which remains bounded for bounded values of $T$) such that
\begin{equation}\label{holderb}
\|\partial_tu\|_{\mathcal{C}^{\alpha,\frac{\alpha}{2s}}(Q_T)} + \|(-\Delta)^su\|_{\mathcal{C}^{\alpha,\frac{\alpha}{2s}}(Q_T)} \leq C(\|u_0\|_{C^{2s+\alpha}(\T^d)}+\|f\|_{\mathcal{C}^{\alpha,\frac {\alpha} {2s}}(Q_T)})\ .
\end{equation}
\end{thm}

We begin with some preliminary decay estimates for the fractional heat semigroup $\mathcal{T}_t$ in H\"older spaces.
\begin{lemma}\label{EstHolFracHeat}
For every $0\leq\theta_1<\theta_2$, $\theta_1,\theta_2\in\R$, there exists $C=C(\theta_1,\theta_2)$ such that for all $f\in C^{\theta_1}(\T^d)$
  \[
  \|\mathcal{T}_tf\|_{C^{\theta_2}(\T^d)}\leq Ct^{-(\theta_2-\theta_1)/2s} \|f\|_{C^{\theta_1}(\T^d)}\ .
  \]
\end{lemma}
\begin{proof}
Computations of Remark \ref{DerivateKernel} (in particular the representation formula for $\mathcal{T}_t$ and Young's inequality for convolution) show that for every $k>h$, $k,h\in\N\cup\{0\}$ there exists $C=C(h,k)$
\[
\|\mathcal{T}_tf\|_{C^{k+h}(\T^d)}\leq Ct^{-\frac{k}{2s}}\|f\|_{C^h(\T^d)}\ .
\]
This implies that $\mathcal{T}_tf:C^{h}(\T^d)\to C^{k+h}(\T^d)$ is bounded for $t > 0$. Recall that, as a consequence of the so-called Reiteration Theorem \cite[Section 1.2.4]{LunardiBook} and \cite[Theorem 1.1.14 and Example 1.1.7]{LunardiNote} (whose proofs can be readily adapted to the torus) we get
\[
(C^h(\T^d),C^{k+h}(\T^d))_{\alpha,\infty}=C^{h+\alpha}(\T^d)\ .
\]
In addition, one also has $\mathcal{T}_tf:L^{\infty}(\T^d)\to L^{\infty}(\T^d)$. By interpolation (see \cite[Proposition 1.2.6]{LunardiBook}), $\mathcal{T}_t$ maps $C^{\theta_1}(\T^d)$ onto $C^{\theta_2}(\T^d)$ with the desired estimate.
\end{proof}
\begin{proof}[Proof of Theorem \ref{Regularity}]
\textit{Step 1}. We first prove the existence of a constant $C>0$ 
such that
\[
\sup_{t\in[0,T]}\|u(\cdot,t)\|_{C^{2s+\alpha}(\T^d)}\leq C(\sup_{t\in[0,T]}\|f(\cdot,t)\|_{C^{\alpha}(\T^d)}+\|u_0\|_{C^{2s+\alpha}(\T^d)})\ .
\]
We first observe that for $s,\alpha\in (0,1)$ such that $2s+\alpha$ is not an integer we have
\[
C^{2s+\alpha}(\T^d)=(C^{\alpha+\delta}(\T^d),C^{2s+\alpha+\delta}(\T^d))_{1-\delta/2s,\infty}\ , \quad 0<\delta<2s\ .
\]
We show that $u(\cdot,t)$ is bounded with values in $C^{2s+\alpha}(\T^d)$. Fix $t\in[0,T]$. Then, for every $\xi>0$ we split $u(t)$ as $u(t)=a(\xi)+b(\xi)+c(\xi)$ using Duhamel's formula, that is 
\begin{align*}
&  a(\xi)=\int_0^{\min\{\xi,t\}}\mathcal{T}_{\tau}f(t-\tau)(x)d\tau, \\ & b(\xi)=\int_{\min\{\xi,t\}}^t\mathcal{T}_{\tau}f(t-\tau)(x)d\tau, \\ & c(\xi)=\mathcal{T}_{t-\min\{\xi,t\}}\mathcal{T}_{\min\{\xi,t\}}u_0 .
\end{align*}
Then $a(\xi)\in C^{\alpha+\delta}(\T^d)$, $b(\xi),c(t)\in C^{2s+\alpha+\delta}(\T^d)$ for each $\delta\in(0,2s)$. Indeed,
\begin{multline*}
\|a(\xi)\|_{C^{\alpha+\delta}(\T^d)} \leq \int_0^{\min\{\xi,t\}}\frac{C}{\tau^{\delta/2s}}d\tau\sup_{\tau\in[0,T]}\|f(\tau)\|_{C^{\alpha}(\T^d)} \\
\leq \frac{C}{1-\delta/2s}\xi^{1-\delta/2s}\sup_{\tau\in[0,T]}\|f(\tau)\|_{C^{\alpha}(\T^d)}\ . 
\end{multline*}
In addition
\begin{multline*}
\|b(\xi)\|_{C^{2s+\alpha+\delta}(\T^d)} \leq \int_{\min\{\xi,t\}}^t\frac{C}{\tau^{1+\delta/2s}}d\tau\sup_{\tau\in[0,T]}\|f(\tau)\|_{C^{\alpha}(\T^d)} \\
\leq \frac{C}{\delta/2s}\xi^{-\delta/2s}\sup_{\tau\in[0,T]}\|f(\tau)\|_{C^{\alpha}(\T^d)}\ . 
\end{multline*}
Similarly to the above computations we have
\[
\|c(\xi)\|_{C^{2s+\alpha+\delta}(\T^d)}\leq \| \mathcal{T}_{\min\{\xi,t\}}u_0 \|_{C^{2s+\alpha+\delta}(\T^d)} \leq C\xi^{-\delta/2s}\|u_0\|_{C^{2s+\alpha}(\T^d)}.
\]
Therefore, by the definition of $K$ in Section \ref{susb} we have
\begin{multline*}
\xi^{-(1-\delta/2s)}K(\xi,u(t),C^{\alpha+\delta}(\T^d),C^{2s+\alpha+\delta}(\T^d))  \\\le \xi^{-(1-\delta/2s)}(\|a(\xi)\|_{C^{\alpha+\delta}(\T^d)} + \xi \|b(\xi) + c(\xi)\|_{C^{2s+\alpha+\delta}(\T^d)})  \\
\le C (\sup_{\tau\in[0,T]}\|f(\tau)\|_{C^{\alpha}(\T^d)} + \|u_0\|_{C^{2s+\alpha}(\T^d)}) \ .
\end{multline*}
This shows in particular that $u(t)\in C^{2s+\alpha}(\T^d)=(C^{\alpha+\delta}(\T^d),C^{2s+\alpha+\delta}(\T^d))_{1-\delta/2s,\infty}$ and
\[
\|u(t)\|_{C^{2s+\alpha}(\T^d)}\leq C(\|f\|_{C_x^{\alpha}(Q_T)} + \|u_0\|_{C^{2s+\alpha}(\T^d)})\ .
\]
for all $t \in [0, T]$. Since $\partial_tu=-(-\Delta)^su+f$ and $\|(-\Delta)^s u(t)\|_{C^\alpha(\T^d)}$ is controlled by $\|u(t)\|_{C^{2s+\alpha}(\T^d)}$ (see, e.g. \cite[Theorem 1.4]{rs}), we obtain the bound on $\|\partial_tu\|_{{C}^{\alpha}(\T^d)}$ + $\|(-\Delta)^s u(t)\|_{C^\alpha(\T^d)}$.
\par\smallskip
\textit{Step 2}. We need to show that $\partial_tu$ and $(-\Delta)^s u$ are both $\alpha/2s$-H\"older continuous in time. Note that as before it is sufficient to estimate the term $(-\Delta)^su$. One can proceed adapting the arguments in \cite[Theorem 4.0.14]{LunardiNote} to the fractional framework, and essentially use estimates of Lemma \ref{EstHolFracHeat}. We refer the reader to \cite{TesiAle} for detailed computations. Anyhow, our setting falls into a general treatment for abstract parabolic equations, see \cite{Sinestrari} or \cite[Theorem 4.0.15]{LunardiNote}.
\end{proof}

Concerning parabolic regularity in Sobolev spaces, we need the following
\begin{thm}\label{fracregtor} Let $p > 1$, $\eps > 0$ and $\mu \in \R$. Suppose that $u \in \mathcal{H}_p^{\mu}(Q_T)$ solves \eqref{fracreghol} with $u_0 \in H^{\mu - 2s/p+\eps}_p(\T^d)$. Then, there exists $C > 0$, that depends on $d, T, p, s, \eps$ (but remains bounded for bounded values of $T$) such that
\[
\|u\|_{\mathcal{H}_p^{\mu}(Q_T)} \le C(\|f\|_{\mathbb{H}_p^{\mu-2s}(Q_T)} + \|u_0\|_{\mu - 2s/p + \eps,p}).
\]
\end{thm}

\begin{proof} A detailed proof of this result on $\R^d$ can be found in \cite{CL}. In the periodic setting one may proceed as follows. Note first that by Duhamel's formula $u(t) = u_1(t) + u_2(t)$, where
\begin{equation*}
u_1(t)=\mathcal{T}_t u_0, \qquad u_2(t) = \int_{0}^t\mathcal{T}_{t-\tau}f(\tau)d\tau.
\end{equation*}

Recall that $(-\Delta)^s$ generates the analytic semigroup $\mathcal{T}_t$ on $L^p(\T^d)$ in view of Remark \ref{anse}. This observation allows to apply the abstract regularity result \cite[Theorem 1]{L}, which yields the following estimate for $u_2$ 
\[
\|u_2\|_{\mathcal{H}_p^{2s}(Q_T)} \le C\|f\|_{L^p(Q_T)}\ . 
\]
The general case follows by the isometry property of the operator $(I-\Delta)^{\frac{\mu}{2}}$ in view of Remark \ref{isometry}.

The estimate of the term involving the initial datum $\|u_1\|_{\mathcal{H}_p^{\mu}(Q_T)} \le C\|u_0\|_{\mu - 2s/p + \eps,p}$ can be obtained directly using decay estimates. We assume without loss of generality $\eps<\frac{2s}{p}$. By Lemma \ref{decaytorus}-(i) we have
\begin{equation*}
\norm{u_1(t)}_{\mu,p} = \norm{\mathcal{T}_t u_0}_{\mu,p} \leq Ct^{-\frac{1}{p}+\frac{\epsilon}{2s}}\norm{u_0}_{\mu-2s/p+\epsilon,p}\ .
\end{equation*}
Note that $C$ here does not depend on $T$. Integrating between 0 and $T$ we have
\begin{equation*}
\norm{u_1}_{\mathbb{H}_p^{\mu}(Q_T)}^p=\int_0^T\norm{u_1(\cdot,t)}^p_{\mu,p}dt\leq C T^{\frac{p \epsilon}{2s}} \norm{u_0}^p_{\mu-2s/p+\epsilon,p}
\end{equation*}
Since $u_1$ solves $\partial_t u_1+(-\Delta)^su_1=0$ we get
\begin{multline*}
\norm{\partial_t u_1}_{\mathbb{H}_p^{\mu-2s}(Q_T)}^p=\int_0^T\norm{\partial_t u_1(\cdot,t)}^p_{\mu-2s,p}=\int_0^T\norm{(-\Delta)^su_1(\cdot,t)}^p_{\mu-2s,p}dt \\
\leq C \int_0^T\norm{(I-\Delta)^su_1(\cdot,t)}^p_{\mu-2s,p}dt= C \int_0^T\norm{u_1(\cdot,t)}^p_{\mu,p}dt\ ,
\end{multline*}
that allows to conclude.
\end{proof}

\begin{rem} In Theorem \ref{fracregtor} we ``pay a price'' of $\eps$ in terms of the regularity of the initial datum in order to make the argument more transparent. Set now $\mu = 2s$ for simplicity. Actually, it is natural for the initial datum $u_0$ in \eqref{fracreghol} to belong to the space $T(p,0,H^{2s}_p(\T^d), L^p(\T^d))$, that is the space of traces of functions in $\mathcal{H}_p^{2s}(\T^d\times(0,+\infty)) = W(p,0,H^{2s}_p(\T^d), L^p(\T^d))$. Similarly, $T(p,0,H^{2s}_p(\T^d), L^p(\T^d))$, is the space of traces of functions in $\mathcal{H}_p^{2s}(Q_T)$ by an easy localization argument. As mentioned in Section \ref{sintro}, $T(p,0,H^{2s}_p(\T^d), L^p(\T^d))$ is equivalent to the real interpolation space $(H^{2s}_p(\T^d), L^p(\T^d))_{1/p, p}$. The latter is isomorphic to the Besov space $B_{p,p}^{2s-2s/p}(\T^d)$ and therefore to $W_p^{2s-2s/p}(\T^d)$.
This chain of equivalences can be motivated as follows: arguing as in \cite[Theorem 6.4.5-(4)]{BL} and \cite[Exercise 6.8.7]{BL}, $(H^{2s}_p(\T^d), L^p(\T^d))_{1/p, p}$ is equivalent to the space of functions $u\in L^p(\T^d)$ with finite seminorm
\[
[u]_{B_{p,p}^{2s-2s/p}(\T^d)}=\left(\iint_{\T^d \times \T^d} \frac{|u(x)-u(x+h)|^p}{|h|^{d+2sp-2s }} dx dh \right)^{\frac1p}<\infty.
\]
This is indeed the Gagliardo seminorm, an equivalent way to define the space $W_p^{2s-2s/p}(\T^d)$ (see \cite[Example 1.0.6]{LunardiNote}). Such space is larger in general than $H^{2s - 2s/p + \eps}_p(\T^d)$ by Lemma \ref{inclusioni}. Therefore, by these ideas one could slightly relax the dependance on $u(0)$ in, e.g., Theorem \ref{Embedding}, Propositions \ref{embe2}, \ref{embe3}, ...

\end{rem}


\medskip
\begin{flushright}
\noindent \verb"cirant@math.unipd.it"\\
Dipartimento di Matematica ``Tullio Levi-Civita''\\ Universit\`a di Padova\\
via Trieste 63, 35121 Padova (Italy)

\noindent \verb"alessandro.goffi@gssi.it"\\
Gran Sasso Science Institute\\
viale Francesco Crispi 7, 67100 L'Aquila (Italy)
\end{flushright}

\end{document}